\documentclass[english,12pt]{article}
\usepackage[latin1]{inputenc}
\usepackage[dvips]{graphics}
\usepackage{amsmath,amsthm,amssymb,babel}

\newtheorem{problem}{Problem}

\newcommand{\R}{{\if mm {\rm I}\mkern -3mu{\rm R}\else \leavevmode
\hbox{I}\kern -.17em\hbox{R} \fi}}
\newcommand{\blambda}{\mbox{\boldmath{$\lambda$}}}

\newcommand{\bmu}{\mbox{\boldmath{$\mu$}}}

\newcommand{\bu}{\mbox{\boldmath{$u$}}}

\newcommand{\bv}{\mbox{\boldmath{$v$}}}
\newcommand{\bw}{\mbox{\boldmath{$w$}}}

\newcommand{\bx}{\mbox{\boldmath{$x$}}}
\newcommand{\bxi}{\mbox{\boldmath{$\xi$}}}

\newcommand{\fb}{\mbox{\boldmath{$f$}}}

\newcommand{\bsigma}{\mbox{\boldmath{$\sigma$}}}
\newcommand{\btau}{\mbox{\boldmath{$\tau$}}}

\newcommand{\bvarepsilon}{\mbox{\boldmath{$\varepsilon$}}}
\newcommand{\bnu}{\mbox{\boldmath{$\nu$}}}

\newcommand{\bzero}{\mbox{\boldmath{$0$}}}

\newcommand{\cS}{\mbox{{${\mathcal S}$}}}

\newcommand{\wu}{\mbox{{${\widetilde{u}}$}}}
\newcommand{\wl}{\mbox{{${\widetilde{\lambda}}$}}}
\newcommand{\wf}{\mbox{{${\widetilde{f}}$}}}
\newcommand{\wh}{\mbox{{${\widetilde{h}}$}}}
\newcommand{\wfb}{\mbox{\boldmath{${\widetilde{f}}$}}}

\newtheorem{theorem}{Theorem}[section]
\newtheorem{lemma}[theorem]{Lemma}

\newtheorem{proposition}[theorem]{Proposition}
\newtheorem{example}[theorem]{Example}
\newtheorem{remark}[theorem]{Remark}

\def\real{\mathbb{R}}

\numberwithin{equation}{section}

\author{\it  Mircea Sofonea$^1$
\ and\ 	Andaluzia Matei$^2$ \\[5mm]
$^1$\,\small{Laboratoire de Math\'ematiques et Physique}\\[-1mm]{\small Universit\'e de
	Perpignan Via Domitia}\\[-1mm] \small{52 Avenue Paul Alduy, 66860 Perpignan,
	France}\\[3mm]
$^2$\,\small{Department of Mathematics,
		University of Craiova}\\[-1mm]
	{\small  A.I.~Cuza Street 13, 200585, Craiova, Romania}}

\title{\bf Convergence and Optimization Results for
	a History-dependent  Variational Problem}

\begin{document}
\maketitle

\vskip3mm
\begin{abstract}
We consider a mixed variational problem in real Hilbert spaces, defined on
on the unbounded interval of time $[0,+\infty)$ and  governed by a
history-dependent operator. We state the  unique solvability of the problem, which follows from a general existence and uniqueness result obtained in
\cite{SM-JOGO}. Then, we state and prove a general convergence result. The proof is based on
arguments of monotonicity, compactness, lower semicontinuity and Mosco convergence.
Finally, we consider a general optimization   problem for which  we prove the existence of minimizers.
The mathematical tools developed  in this paper are useful in the analysis  of a large class of nonlinear  boundary value problems   which, in a weak formulation, lead to  history-dependent mixed variational problems. To provide an example, we illustrate our abstract results in the study of a frictional contact  problem for viscoelastic materials with long memory.
\end{abstract}

\medskip \noindent {\bf AMS Subject Classification\,:}  35M86, 35M87, 49J40, 74M15, 74M10.

\bigskip \noindent {\bf Key words\,:}
history-dependent operator, mixed variational problem, Lagrange
multiplier, Mosco convergence, pointwise convergence,  optimization problem, viscoelastic material, frictional contact.

\vskip15mm

\section{Introduction} \setcounter{equation}0

Mixed variational problems involving Lagrange multipliers provide a useful framework in which a large number
of problems with or without unilateral constraints can be cast and can
be solved numerically. They are intensively used in Solid and Contact Mechanics as well as in many engineering
applications.
Existence and uniqueness results in the study of stationary
mixed variational problems can be found in \cite{Cea,ET,HHN,HHNL,LGT}, for instance.
Recently, there is an interest in the study of time-dependent mixed variational problems involving a special case of operators, the so-caled history-dependent operators. Such kind of operators arise in Solid and Contact Mechanics and describe memory effects in both the constitutive law and the interface boundary conditions. Reference in the field are \cite{SM11, SM, SMBOOK}.
The analysis of
 various mixed variational problems associated to mathematical models which describe the contact between a deformable  body and a foundation
 can be found in \cite{HR, HMW, HW, BRB,ANZIAM, AOSR,MMS} and,
 more recently, in \cite{AHLMR,BMS12-1, BMS12-2,SMX}.

The current paper represents a continuation of our previous paper \cite{SM-JOGO}. There, we
considered a new class of mixed variational problems with
history-dependent operators and used arguments of saddle point and fixed point in order to
prove its unique solvability.  Here,  we consider a special case of the history-dependent mixed variational problems considered in
\cite{SM-JOGO} for which we provide aditional results.
Everywhere in this paper we assume that $X$, $Y$ and $Z$ are real Hilbert spaces endowed with the inner products
$(\cdot,\cdot)_X$, $(\cdot,\cdot)_Y$ and $(\cdot,\cdot)_Z$. The associated norms will be denoted by $\|\cdot\|_X$, $\|\cdot\|_Y$ and $\|\cdot\|_Z$, respectively. Moreover, $0_X$ and $0_Y$ will represent the zero elements of the spaces $X$ and $Y$, and $X\times Y$ is their product space endowed with the canonical inner product. A typical element of $X\times Y$ will be denoted by $(u,\lambda)$.
We also denote by $\mathbb R_+$ the set of nonegative real
numbers, i.e. $\mathbb R_+=[0,+\infty)$, and we use the notation
$C(\mathbb{R}_+;X)$, $C(\mathbb{R}_+;Y)$  and $C(\mathbb{R}_+;Z)$ for the space of
continuous functions defined on $\mathbb{R}_+$ with values in $X$, $Y$
and $Z$, respectively.
Consider three operators $A:X\to X$, $\cS:C(\mathbb R_+;X)\to C(\mathbb R_+;X)$
and $\pi:X\to Z$, a form $b:X\times Y\to\real$, a set $\Lambda\subset Y$ and two functions  $f:\real_+\to Z$,
$h:\real_+\to X$.
With these data we introduce the following problem.

\medskip
\begin{problem}\label{p} Find the functions $u:\mathbb R_+\to X$ and
$\lambda:\mathbb R_+\to \Lambda$ such that
\begin{eqnarray}
&&\hspace{-5mm}(Au(t),v)_X+(\cS u(t),v)_X+b(v,\lambda(t))=(f(t),\pi v)_Z
\qquad \forall\, v\in X,\label{1}\\[2mm]
&&\hspace{-5mm}\qquad b(u(t),\mu-\lambda(t))\leq
b(h(t),\mu-\lambda(t)) \qquad\forall\, \mu\in
\Lambda,\label{2}
 \end{eqnarray}
for all $t\in \mathbb R_+$.
\end{problem}

\noindent
If $\Lambda\subset Y$ is unbounded, the unique solvability of Problem \ref{p} is a direct consequence of the  existence and uniqueness result obtained in \cite{SM-JOGO}. However, in the present paper we remove this restriction, as explained in the next section.

The current paper has three main objectives. The first one is to  provide  a
continuous dependence result for the solution  to Problem \ref{p} with respect to the data. The second one is to study an optimization problem related to this history-dependent mixed problem. Finally, the third objective is to illustrate how these abstract results can be used in the analysis of  nonlinear boundary value problems which describe the contact of a deformable body with a foundation.

The rest of the paper is structured as follows. In Section~\ref{s2} we present some preliminary material, including a general existence and uniqueness result
for mixed problems of the form (\ref{1})--(\ref{2}). In Section
\ref{s3} we state and prove our main result, Theorem \ref{t1}, which concerns the pointwise convergence  of the solution to Problem \ref{p} with respect to $A$, $\cS$, $b$, $\Lambda$, $f$ and $h$. Its proof is based on arguments of monotonicity, compactness, lower semicontinuity and Mosco convergence.
Then, in  Section \ref{s4} we apply this  convergence result in the study of an optimization problem associated to Problem \ref{p}, for which we prove the existence of minimizers.
Finally, in Section \ref{s5} we consider a mathematical model which describes the frictional contact of a viscoelastic
body with an obstacle. In a variational formulation,  this problem leads to a history-dependent mixed problem of the form (\ref{1})--(\ref{2}). We illustrate how our existence, uniqueness and convergence result can be applied in the analysis  of this nonlinear problem.

\section{Preliminary results}\label{s2}

In this  section we present some preliminary results useful in the study of Problem \ref{p} and, to this end, we consider the
following assumptions.

\begin{eqnarray}
&&\left\{ \begin{array}{ll} A:X\to X\ \mbox{and, moreover:}\\[3mm]
{\rm (a)\ there\ exists\ } m_A>0\ {\rm such\  that}\\
\qquad (Au-Av,\,u-v)_X\geq m_A\|u-v\|_X^2\,\quad \forall\,
u,v\in X;\\[2mm]
{\rm (b)\ there\ exists\ } L_A>0\ {\rm such\  that}\\
\qquad \|Au-Av\|_X\leq L_A\, \|u-v\|_X\,\quad \forall\, u,v\in X.
\end{array}\right. \label{A}
\\[6mm]
&&\left\{ \begin{array}{l}
\cS:C(\mathbb R_+;X)\to C(\mathbb R_+;X)
\mbox{ and for\ each}\ m\in\mathbb{N}\\[3mm]
\ {\rm there\ exists}
\mbox{ and } s_m\geq 0{\rm\ such\ that}\\[3mm]
\quad\|{\cS}u_1(t)-{\cS}u_2(t)\|_X\le s_m
\displaystyle{\int_0^t\|u_1(s)-u_2(s)\|_X\,ds}\ \\[4mm]
\qquad\forall\, u_1,\,u_2\in C(\mathbb R_+;X), \
t\in[0,m].
\end{array}\right. \label{S}
\end{eqnarray}
\begin{eqnarray}
&&\left\{ \begin{array}{ll}
b:X\times Y\to {\mathbb R}
\mbox{ is a bilinear form and, moreover:}\\[2mm]
{\rm (a)\ there\ exists\ } M_b>0\ {\rm such\  that}\\
\qquad |b(v,\mu)|\leq
M_b\|v\|_X\|\mu\|_{Y}\,\quad \forall\,v\in X,\,\mu\in Y;\\[3mm]
{\rm (b)\ there\ exists\ } \alpha_b>0\ {\rm such\  that}\\[1mm]
\qquad\displaystyle\inf_{\mu\in Y, \mu\neq 0_{Y}}\,\sup_{v\in
	X,v\neq 0_X}\,\frac{b(v,\mu)}{\|v\|_X\|\mu\|_{Y}}\geq \alpha_b.
\end{array}\right. \label{b}
\end{eqnarray}

\smallskip
\begin{equation}\label{lam} \Lambda \text{ is a closed convex subset of } Y \text{ that contains }0_{Y}.
\end{equation}

Note that assumption (\ref{A}) shows that the operator $A$ is strongly monotone and Lipschitz continuous. Moreover, following the terminology
introduced in \cite{SM11} and used in a large number of papers, assumption (\ref{S}) shows that $\cS$ is a history-dependent operator.
Finally,  (\ref{b})(a) shows that the bilinear form $b$ is continuous and  (\ref{b})(b) represents the well-known ``inf-sup" condition.

We denote in what follows  by $X\times \Lambda$ the  product of $X$ and $\Lambda$ and we use the notation $C(\mathbb R_+;\Lambda)$,
$C(\mathbb R_+;X\times\Lambda)$ for the set of functions defined on $\mathbb R_+$ with values in $\Lambda$ and $X\times\Lambda$, respectively.
In addition, we recall the following existence and uniqueness results.

\begin{theorem}\label{t0} Assume $(\ref{A})$, $(\ref{b})$, $(\ref{lam})$. Then,
given $g,\,k\in X$, there exists a
	unique pair $(u,\lambda)\in X\times \Lambda$ such that
	\begin{eqnarray}
	&&(Au,v)_X+b(v,\lambda)=(g,\,v)_X\qquad
	\forall\, v\in X,\label{3}\\[2mm]
	&&\quad b(u,\mu-\lambda)\leq b(k,\mu-\lambda)\qquad\forall\,\mu\in
	\Lambda.\label{4}
	\end{eqnarray}

\end{theorem}

\begin{theorem}\label{t00} Assume $(\ref{A})$--$(\ref{lam})$. Then, given
$g,\,k\in C(\real_+,X)$, there exists a unique pair $(u,\lambda)\in C(\mathbb R_+;X\times\Lambda)$ such that
\begin{eqnarray}
&&\hspace{-5mm}(Au(t),v)_X+(\cS u(t),v)_X+b(v,\lambda(t))=(g(t),v)_X
\qquad \forall\, v\in X,\label{5}\\[2mm]
&&\hspace{-5mm}\qquad b(u(t),\mu-\lambda(t))\leq
b(k(t),\mu-\lambda(t)) \qquad\forall\, \mu\in
\Lambda.\label{6}
\end{eqnarray}
	
\end{theorem}

Theorem \ref{t0} corresponds to Theorem 5.2 in \cite{ANZIAM} for $\Lambda$ unbounded and to Theorem 2.1 in \cite{Ma13} for $\Lambda$ bounded. Its proof is based on arguments of saddle points and the Banach fixed point principle.
The proof of Theorem \ref{t00} can be carried out in several steps,  based on Theorem \ref{t0} combined with a fixed point argument for history-dependent operators proved in \cite{SAM}.  Recall that Theorem  \ref{t00} represents a particular case of Theorem 2.1 in \cite{SM-JOGO}, where the operator $\cS$ was assumed to depend on both the unknowns $u$ and $\lambda$, and $\Lambda$ was supposed to be  unbounded. Nevertheless, by a slight modification of the proof of Theorem 2.1 in \cite{SM-JOGO}, it follows that Theorem \ref{t00} still remains valid if $\Lambda$ is bounded.

Consider now the following additional assumption.
\begin{equation}\label{pi}
 \pi:X\to Z\ \ {\rm is\ a\ linear\ continuous\ operator}
\end{equation}
which implies that there exists $c_0>0$ such that
\begin{equation}\label{pibis}
\ \|\pi v\|_{Z} \leq c_0\,\|v\|_X\quad \forall\,v\in X.
\end{equation}

We complete Theorem \ref{t0} with the following existence, uniqueness and continuous dependence result.


\medskip
\begin{proposition}\label{pr1} Assume $(\ref{A})$, $(\ref{b})$, $(\ref{lam})$, $(\ref{pi})$. Then,
	given $\eta,\, k\in X$ and $f\in Z$ there exists a
	unique pair $(u,\lambda)\in X\times \Lambda$ such that
	\begin{eqnarray}
	&&(Au,v)_X+(\eta,v)_X+b(v,\lambda)=(f,\,\pi v)_Z\qquad
	\forall\, v\in X,\label{7}\\[1mm]
	&&\quad b(u,\mu-\lambda)\leq b(k,\mu-\lambda)\qquad\forall\,\mu\in
	\Lambda.\label{8}
	\end{eqnarray}
	In addition, if  $(u_1,\,\lambda_1)$ and $(u_2,\,\lambda_2)$ are
	the solutions of the problem $(\ref{7})$--$(\ref{8})$
	corresponding to the data $\eta_1,\, k_1\in X$, $f_1\in Z$ and $\eta_2,\,k_2\in X$, $f_2\in Z$,
	respectively, then there exists $d_0>0$ which depends only on $m_A$, $L_A$, $M_b$, $\alpha_b$ and $c_0$
	such that
	\begin{equation}\label{9}
	\|u_1-u_2\|_{X}+\|\lambda_1-\lambda_2\|_{Y}\leq
	d_0(\|\eta_1-\eta_2\|_{X}+\|f_1-f_2\|_Z+\|k_1-k_2\|_{X}).
	\end{equation}
\end{proposition}

\begin{proof}
We use assumption (\ref{pi}) and the Riesz representation theorem to define the element $g\in X$ by the equality
	\begin{equation*}
	(g,v)_X=(f,\pi v)_Z-(\eta,v)_X\qquad\forall\, v\in X.
	\end{equation*}
	The existence and uniqueness part in Proposition \ref{pr1} is now a direct consequence of Theorem \ref{t0}.
	
Denote in what follows by $(u_i,\,\lambda_i)$
the solution of the mixed problem $(\ref{7})$--$(\ref{8})$ for the data $\eta=\eta_i,\, k=k_i\in X$, $f=f_i\in Z$, for $i=1,2$.	 Let $v\in X$.
Then, using
	(\ref{7}) it follows that
	\begin{equation}\label{a1}
	(Au_1-Au_2,v)_X+b(v,\lambda_1-\lambda_2)=(f_1-f_2,\pi v)_Z-(\eta_1-\eta_2,v)_X
	\end{equation}
	and, using (\ref{A})(b),  (\ref{pi}) we find that
	\[b(v,\lambda_1-\lambda_2)\le c_0\|f_1-f_2\|_Z\|v\|_X+\|\eta_1-\eta_2\|_X\|v\|_X+L_A\|u_1-u_2\|_X\|v\|_X.\] We now use (\ref{b})(b) and the previous
	inequality to obtain that
	\begin{equation}\label{a2}
	\alpha_b\,\|\lambda_1-\lambda_2\|_{Y}\leq
	c_0\|f_1-f_2\|_Z+\|\eta_1-\eta_2\|_X+L_A\|u_1-u_2\|_X.
	\end{equation}
	
	On the other hand, (\ref{8}) yields
	\begin{equation*}
	b(u_1-u_2,\lambda_2-\lambda_1)\leq b(k_1-k_2,\lambda_2-\lambda_1)
	\end{equation*}
	and, therefore, using condition (\ref{b})(a) we find that
	\begin{equation}\label{a3}
	b(u_1-u_2,\lambda_2-\lambda_1)\leq M_b\,\|k_1-k_2\|_X
	\|\lambda_1-\lambda_2\|_{Y}.
	\end{equation}
	We now take $v=u_1-u_2$ in (\ref{a1}) and use (\ref{a3}) in the
	resulting inequality to deduce that
	\vspace{-2mm}
	\begin{eqnarray*}
	&&(Au_1-Au_2,u_1-u_2)_X\le (f_1-f_2,\pi u_1-\pi u_2)_Z\\ [0mm]
	&&\qquad-(\eta_1-\eta_2,u_1-u_2)_X+M_b\,\|k_1-k_2\|_X
	\|\lambda_1-\lambda_2\|_{Y}.
	\end{eqnarray*}
	
	\noindent Therefore, using the assumptions (\ref{A})(a) and (\ref{pi}) it follows
	that
	\begin{eqnarray}
	&&\label{a4}
	m_A\|u_1-u_2\|_X^2\le c_0\|f_1-f_2\|_Z\|u_1-u_2\|_X\\[0mm]
	&&\qquad+\|\eta_1-\eta_2\|_X\|u_1-u_2\|_X+M_b\|k_1-k_2\|_X
	\|\lambda_1-\lambda_2\|_{Y}.\nonumber
	\end{eqnarray}
	We now use  (\ref{a4}) and (\ref{a2}) together with the elementary inequalities
	\[ab\le \frac{a^2}{2c}+\frac{cb^2}{2}\,,\quad (a+b+c)^2\le 3(a^2+b^2+c^2)\qquad\forall\, a,\ b,\ c>0\]
	to see that\vspace{-1mm}
	\begin{eqnarray}
	&&m_A\|u_1-u_2\|_X^2\leq
	\frac{c_0^2\|f_1-f_2\|^2_X}{2c_1}+c_1\|u_1-u_2\|_X^2\label{a5}\\[1mm]
	&&\qquad+\frac{\|\eta_1-\eta_2\|_X^2}{2c_1}+\frac{M_b^2\|k_1-k_2\|_X^2}{2c_2}+\frac{c_2\|\lambda_1-
		\lambda_2\|_{Y}^2}{2}\,,\nonumber\\[4mm]
	&&\|\lambda_1-\lambda_2\|^2_{Y}\leq
	\frac{3}{\alpha_b^2}\big(c_0^2\|f_1-f_2\|^2_Z+\|\eta_1-\eta_2\|_X^2+L_A^2\|u_1-u_2\|_X^2\big)\,\label{a6}
	\end{eqnarray}
	
	\noindent where $c_1,c_2$ are arbitrary positive
	constants. We now substitute  (\ref{a6}) in (\ref{a5}) and choose $c_1$ and
	$c_2$ such that
	\begin{eqnarray*}
		m_A-{c_1}-\frac{3c_2L_A^2}{2\alpha_b^2}>0.
	\end{eqnarray*}
	In this way we deduce that there exists $c_3>0$,
	which depends only on $m_A$, $L_A$, $M_b$, $\alpha_b$ and $c_0$  such that
	\begin{equation}\label{a7}
	\|u_1-u_2\|_X^2\leq c_3(\|f_1-f_2\|^2_Z+\|\eta_1-\eta_2\|^2_X+\|k_1-k_2\|^2_X).
	\end{equation}
	Finally, using (\ref{a6}) and (\ref{a7}), after some algebra we obtain
	(\ref{9}) with $d_0>0$ depending only on $m_A$, $L_A$, $M_b$, $\alpha_b$ and $c_0$, which concludes the proof.
\end{proof}

Next, we introduce the following assumptions on the data $f$ and $h$ of Problem \ref{p}.
\begin{eqnarray}
&&\label{f}f\in C(\mathbb R_+;Z),\\ [0mm]
&&\label{h} h\in C(\mathbb R_+;X).
\end{eqnarray}
We have the following existence and uniqueness result.

\begin{proposition}\label{pr2} Assume $(\ref{A})$--$(\ref{lam})$, $(\ref{pi})$, $(\ref{f})$ and  $(\ref{h})$.
	Then Problem
	$\ref{p}$ has a unique solution $(u,\lambda)$. Moreover, the
	solution satisfies $(u,\lambda)\in C(\mathbb R_+; X\times\Lambda)$.
\end{proposition}

\begin{proof}
We use assumption (\ref{pi}) and the Riesz representation theorem to define the function $g:\real_+\to X$ by equality
\begin{equation*}
(g(t),v)_X=(f(t),\pi v)_Z\qquad\forall\, v\in X,\ t\in\real_+.
\end{equation*}
Then, we use assumption (\ref{f}) to see that  $g\in C(\mathbb R_+;X)$. Proposition \ref{pr2} is now a direct consequence of Theorem \ref{t00}.
\end{proof}

We end this section by recalling  the following version of the Weierstrass theorem.

\begin{theorem}\label{tW}
	Let $(X,\|\cdot\|_X)$ be a reflexive Banach space, $K$  a nonempty weakly closed
	subset of $X$ and   $J:X\to\mathbb{R}$ a weakly lower semicontinuous function. In addition, assume that either $K$ is bounded or $J$ is coercive,
	i.e., $J(v)\to\infty$ as $\|v\|_X\to\infty$.
	Then, there exists at least one element $u$ such that
	\begin{equation}\label{04}
	u\in K,\qquad J(u)\le J(v)\qquad\forall\, v\in K.
	\end{equation}
\end{theorem}

The proof of Theorem \ref{tW}  is based on standard arguments  which can be found in many books and survey as, for instance, \cite{Ku,SM}.

\section{A convergence result}\label{s3}

The solution $(u,\lambda)$ obtained in
Proposition \ref{pr2} depends on $A$, $\cS$, $b$, $\Lambda$, $f$ and  $h$. In this section we state and prove its convergence with respect to these data, which  represents a crucial ingredient in the study of the optimization  problem we shall consider in Section \ref{s4}. Unless stated otherwise, all the sequences we introduce below are indexed upon $n\in\mathbb{N}$ and all the limits, upper and lower limits  are considered as $n\to\infty$, even if we do not mention it explicitly. The symbols ``$\rightharpoonup$"  and ``$\to$"
denote the weak and the strong convergence in various spaces which will be specified. Nevertheless, for simplicity, we write $g_n\to g$ for the convergence in $\mathbb{R}$.

The functional framework is as follows.
For each $n\in\mathbb{N}$ we consider two operators $A_n$ and $S_n$, a form $b_n$, a set $\Lambda_n$ and  two functions $f_n$ and $h_n$
which satisfy assumptions (\ref{A}), (\ref{S}), (\ref{b}), (\ref{lam}), (\ref{f}) and (\ref{h}), respectively, with constants $m_{n}$, $L_n$, $s_m^n$, $M_n$, $\alpha_n$. To avoid any confusion, when used with $n$, we refer to these  assumptions as assumptions $(\ref{A})_n$, $(\ref{S})_n$, $(\ref{b})_n$, $(\ref{lam})_n$ $(\ref{f})_n$ and $(\ref{h})_n$.
Then, if condition (\ref{pi}) is satisfied, we deduce from  Proposition \ref{pr2} that for each $n\in\mathbb{N}$ there exists a unique solution $(u_n,\lambda_n)\in C(\real_+;X\times\Lambda_n)$ for the following mixed variational problem.
\begin{problem}\label{pn} Find the functions $u_n:\mathbb R_+\to X$ and
	$\lambda_n:\mathbb R_+\to \Lambda_n$ such that
	\begin{eqnarray}
&&\hspace{-5mm}(A_nu_n(t),v)_X+(\cS_n u_n(t),v)_X+b_n(v,\lambda_n(t))=(f_n(t),\pi v)_Z
\qquad \forall\, v\in X,\label{5nn}\\[0mm]
&&\hspace{-5mm}\qquad b_n(u_n(t),\mu-\lambda_n(t))\leq
b_n(h_n(t),\mu-\lambda_n(t)) \qquad\forall\, \mu\in
\Lambda_n\label{6nn}
\end{eqnarray}
for all $t\in \mathbb R_+$.
\end{problem}

\medskip
We now  consider the following additional assumptions.
\begin{eqnarray}
	&&\label{cvA}
	\left\{
	\begin{array}{ll}
		\mbox{For each}\ n\in\mathbb{N}\ \mbox{there exist }F_n\ge 0\ \mbox{and}\ \delta_n\ge 0\  \mbox{ such
			that
		}\\[2mm]
		\mbox{(a) } \|A_{n}v-Av\|_X\le F_n(\|v\|_X+\delta_n)\quad\forall\, v\in X;
		\\[3mm]
		\mbox{(b) }\displaystyle\lim_{n\to\infty}F_n=0;\\[3mm]
		\mbox{(c) the sequence}\ \{\delta_n\}\subset\mathbb{R}\ \mbox {is bounded}.
	\end{array}
	\right.\\ [3mm]
	&&\label{bA}\mbox{There exist}\ m_0,\, L_0>0\ \mbox{such that}\ m_n\ge m_0,\ L_n\le L_0\ \ \forall\, n\in\mathbb{N}.
	\\[3mm]
	&&\label{cvS}
		\left\{
		\begin{array}{ll}
		\mbox{For each}\ n,\,m\in\mathbb{N}\ \mbox{there exist }F_n^m\ge 0\ \mbox{and}\ \delta_n^m\ge 0\  \mbox{ such
			that
		}\\[2mm]
		\mbox{(a) } \|S_{n}v(t)-Sv(t)\|_X\le F_n^m(\displaystyle\max_{s\in[0,m]}\|v(s)\|_X+\delta_n^m)\\[2mm]
		\qquad \mbox{for all}\ v\in C(\real_+;X),\ t\in[0,m];
		\\[3mm]
		\mbox{(b) }\displaystyle\lim_{n\to\infty}F_n^m=0,\ \ \forall\,m\in\mathbb{N};\\[3mm]
		\mbox{(c) the sequence}\ \{\delta_n^m\}\subset\mathbb{R}\ \mbox {is bounded},\ \ \forall\,m\in\mathbb{N}.
		\end{array}
		\right.\\ [3mm]
		&&\label{Z}\mbox{For each $m\in\mathbb{N}$ there exists}\ s_m^0>0\ \mbox{such that}\ s_m^n\le s_m^0\ \ \forall\, n\in\mathbb{N}.
	\\[3mm]
	&&\label{cvb}
	\left\{\begin{array}{ll}\mbox{For all sequences}\ \{z_n\}\subset X,\ \{\mu_n\}\subset Y \  \mbox{such that}\\[2mm]
		z_n\rightharpoonup z  \ \mbox{in}\ X,\ \mu_n\rightharpoonup \mu\ \mbox{in}\ Y,\ \mbox{we have}\hspace{26mm} \\ [2mm]
		\limsup\,b_n(w-z_n,\mu_n)\le b(w-z,\mu)\quad\forall\, w\in X.
	\end{array}\right.\\ [3mm]
	&&\label{bb}\mbox{There exist}\ \alpha_0,\, M_0>0\ \mbox{such that}\ \alpha_n\ge \alpha_0,\ M_n\le M_0\ \ \forall\, n\in\mathbb{N}.
	\\[4mm]
	&&\label{cvpi}
	\left\{\begin{array}{ll}\mbox{For all sequences}\ \{v_n\}\subset X\  \mbox{such that}\\[2mm]
		v_n\rightharpoonup v\ \mbox{in}\ X,\  \mbox{we have}
		\quad \pi v_n\to\pi v\quad{\rm in}\quad Y.
	\end{array}\right.\\[3mm]
	&&\label{cvLa}\left\{
	\begin{array}{l}
		\{\Lambda_n\}\ \mbox{\rm converges to}\ \Lambda\ \mbox{\rm in the sense of Mosco, i.e.:}
		\\[2mm]
		\mbox{(a) }
		\mbox{for each $\mu\in \Lambda$ there exists a sequence $\{\mu_n\}$ such that}\\[1mm]
		\qquad\mbox{$\mu_n\in \Lambda_n$  $\forall\,n\in \mathbb{N}$  and $\mu_n\to \mu$ in $Y$;}
		\\[2mm]
		\mbox{(b) }
		\mbox{for each sequence $\{\mu_n\}$ such that}\\[1mm]
		\qquad\mbox{$\mu_n\in \Lambda_n$ $\forall\,n\in \mathbb{N}$  and $\mu_n\rightharpoonup \mu$ in $Y$, we have}\ \mu\in \Lambda.
	\end{array}\right.
	\\[3mm]
&&\label{cvvf}\left\{\begin{array}{ll}
	{\rm (a)}\ f_{n}(t)\rightharpoonup f(t)\quad {\rm in}\quad Y\quad{\rm as}\quad n\to\infty,\ \ \forall\,t\in\mathbb{R}_+;\\[2mm]
	{\rm (b)}\ 	
	\mbox{\rm for each}\ m\in\mathbb{N}\ \mbox{\rm there exists}\ \omega_m>0\  \mbox{\rm such that}\\ [2mm]
	\qquad\|f_n(t)\|_Y\le \omega_m\ \ \forall\, t\in[0,m],\ n\in\mathbb{N}.\\[2mm]
	\end{array}\right.\\[3mm]
&&\label{cvvh}\left\{\begin{array}{ll}
{\rm (a)}\ h_{n}(t)\rightharpoonup h(t)\quad {\rm in}\quad Y\quad{\rm as}\quad n\to\infty,\ \ \forall\,t\in\mathbb{R}_+;\\[2mm]
{\rm (b)}\ 	
\mbox{\rm for each}\ m\in\mathbb{N}\ \mbox{\rm there exists}\ \xi_m>0\  \mbox{\rm such that}\\ [2mm]
\qquad\|h_n(t)\|_Y\le \xi_m\ \ \forall\, t\in[0,m],\ n\in\mathbb{N}.\\[3mm]
\end{array}\right.
\end{eqnarray}

\medskip
Note that assumption (\ref{cvpi}) shows that the linear operator $\pi:X\to Y$ is completely continuous. Details on the convergence of sets in the sense of Mosco used in condition (\ref{cvLa}) can be found in \cite{Mosco}. Such a convergence was used in the recent papers \cite{S1,XS3} in the study of convergence results for elliptic and history-dependent variational-hemivariational inequalities, respectively.

\medskip
\begin{remark}\label{R1}
	It follows from
	assumptions $(\ref{A})_n$ and $(\ref{bA})$ that for each $n\in\mathbb{N}$ the operator $A_n$ satisfies condition $(\ref{A})$ with the constants  $m_0$ and $L_0$.  On the other hand,
	assumptions $(\ref{S})_n$ and $(\ref{Z})$ show that for each $n\in\mathbb{N}$ the operator $\cS_n$ satisfies condition $(\ref{S})$ with the constant  $s_m^0$.
	Finally, assumptions $(\ref{b})_n$ and $(\ref{bb})$ show that for each $n\in\mathbb{N}$ the bilinear form $b_n$ satisfies condition $(\ref{b})$ with the constants  $\alpha_0$ and $M_0$.
\end{remark}

The main result of this section is the following.

\begin{theorem}\label{t1}  Assume  $(\ref{A})$--$(\ref{lam})$,  $(\ref{pi})$,  $(\ref{f}),$
$(\ref{h})$ and, for each $n\in\mathbb{N}$, assume $(\ref{A})_n$--$(\ref{lam})_n$, $(\ref{f})_n$, $(\ref{h})_n$.
Moreover, assume $(\ref{cvA})$--$(\ref{cvvh})$ and denote by $(u_n,\lambda_n)$ and $(u,\lambda)$ the solutions of Problems $\ref{pn}$ and $\ref{p}$, respectively. Then, for all $t\in\real_+$
the following convergences hold:
	\begin{eqnarray}
		&&\label{cvu}u_n(t)\to u(t)\qquad{\rm in}\quad X,\\ [2mm]
		&&\label{cvla}\lambda_n(t)\rightharpoonup \lambda(t)\qquad{\rm in}\quad Y.
	\end{eqnarray}
\end{theorem}

The proof of Theorem \ref{t1} will be carried out in several
steps that we present in what follows. To this end, below in this section we assume that the hypotheses of Theorem \ref{t1} are satisfied
and, for each $n\in\mathbb{N}$, we consider the following auxiliary problems.
\begin{problem}\label{pm} Find  $u_n^0\in X$ and
	$\lambda_n^0\in\Lambda_n$ such that
	\begin{eqnarray}
	&&(A_nu_n^0,v)_X+b_n(v,\lambda_n^0)=0
	\qquad \forall\, v\in X,\label{5m}\\[2mm]
	&&b_n(u_n^0,\mu-\lambda_n^0)\leq 0\qquad\forall\, \mu\in
	\Lambda_n.\label{6m}
	\end{eqnarray}
\end{problem}

\begin{problem}\label{pq} Find the functions $\wu_n:\mathbb R_+\to X$ and
	$\wl_n:\mathbb R_+\to \Lambda_n$ such that
	\begin{eqnarray}
	&&\hspace{-5mm}(A_n\wu_n(t),v)_X+(\cS u(t),v)_X+b_n(v,\wl_n(t))=(f_n(t),\pi v)_Z
	\qquad \forall\, v\in X,\label{5n}\\[2mm]
	&&\hspace{-5mm}\qquad b_n(\wu_n(t),\mu-\wl_n(t))\leq
	b_n(h_n(t),\mu-\wl_n(t)) \qquad\forall\, \mu\in
	\Lambda_n\label{6n}
	\end{eqnarray}
	for all $t\in \mathbb R_+$.
\end{problem}

The first step is given by the
following result.

\medskip
\begin{lemma}\label{l1} For each $n\in\mathbb{N},$ Problem $\ref{pm}$ has a unique solution $(u_n^0,\lambda_n^0)\in X\times\Lambda_n$. Moreover, there exists $a_0>0$ such that
	\begin{equation}\label{x1}
	\| u_n^0\|_X\le a_0,\qquad \| \lambda_n^0\|_Y\le a_0
\qquad\forall\, n\in\mathbb{N}.
	\end{equation}

\end{lemma}

\begin{proof}  The existence and uniqueness part is a direct consequence of Theorem \ref{t0}.
Let $n\in\mathbb{N}$. We use assumption $(\ref{lam})_n$ and test in (\ref{6m}) with $\mu=0_Y$ to obtain that
	 \[b_n(u_n^0,\lambda_n^0)\ge 0.\]
	 We now take $v=u_n^0$ in  (\ref{5m}) and use the previous inequality to see that
	 \[(A_nu_n^0,u_n^0)_X\le0.\]
	 Next, we write $A_nu_n^0=A_nu_n^0-A_n0_X+A_n0_X$ and use assumption $(\ref{A})_n(a)$  to deduce that
	 \begin{equation}\label{10}
	 m_n\| u_n^0\|_X\le \|A_n0_X\|_X.
	 \end{equation}
	
	 On the other hand,  writing $A_n0_X=A_n0_X-A0_X+A0_X$ and using inequality (\ref{cvA})(a) yield
	 \begin{equation}\label{11}
	 \|A_n0_X\|_X\le F_n\delta_n+\|A0_X\|_X.
	 \end{equation}
	 We now combine inequalities (\ref{10}) and (\ref{11}), and use assumption (\ref{bA}) to see that
	 \begin{equation*}
	 \| u_n^0\|_X\le\frac{1}{m_0} \big(F_n\delta_n+\|A0_X\|_X\big).
	 \end{equation*}
	 Finally, we use assumptions (\ref{cvA})(b),(c) to deduce that  the sequence $\{ u_n^0 \}$ is bounded in $X$, i.e., there exists $K_0>0$ which does not depend on $n$ such that
	 \begin{equation}\label{bu}
	 \| u_n^0\|_X\le K_0.
	 \end{equation}

	 Next, we establish the boundedness of $\{ \lambda_n \}$ in $Y$. To this end we use (\ref{5m})  to see that
	 \[b_n(v,\lambda_n^0)=-(A_nu_n^0,v)_X\le\|A_nu_n^0\|_X\|v\|_X\]
	 for all $v\in X$,  which implies that
	 \[ \displaystyle\,\sup_{v\in X,v\neq 0_X}\,\frac{b_n(v,\lambda_n^0)}{\|v\|_X\|\lambda_n^0\|_{Y}}\le
	 \frac{1}{\|\lambda_n^0\|_Y}\|A_nu_n^0\|_X,\]
	 if $\lambda_n^0\ne 0_Y$. Therefore,
	 \[ \displaystyle\inf_{\mu\in Y,
	 	\mu\neq 0_{Y}}\,\sup_{v\in X,v\neq 0_X}\,\frac{b_n(v,\mu)}{\|v\|_X\|\mu\|_{Y}}\le
	 \frac{1}{\|\lambda_n^0\|_Y}\|A_nu_n^0\|_X,\]
	 if $\lambda_n^0\ne 0_Y$.  We now use assumption $(\ref{b})_n$
	 on the bilinear form $b_n$ together with the bound (\ref{bb}) to deduce that
	 \begin{equation}\label{13}
	 \alpha_0\|\lambda_n^0\|_Y\le  \|A_nu_n^0\|_X,
	 \end{equation}
	 both when $\lambda_n^0\ne 0_Y$ and  $\lambda_n^0=0_Y$.
	
	 Next, we use assumptions (\ref{cvA})(a) and (\ref{A})(b)
	 to see that
	 \begin{eqnarray*}
	 	&&\|A_nu_n^0\|_X\le \|A_nu_n^0-Au_n^0\|_X+\|
	 	Au_n^0\|_X\\ [2mm]
	 	&&\quad\le
	 	F_n(\|u_n^0\|_X+\delta_n)+\|Au_n^0-A0_X\|_X+\|A0_X\|_X\\ [2mm]
	 	&&\qquad\le
	 	F_n(\|u_n^0\|_X+\delta_n)+L_A\|u_n^0\|_X+\|A0_X\|_X.
	 \end{eqnarray*}
	 Therefore, by assumptions (\ref{cvA})(b), (c) and the bound (\ref{bu}) we find that the sequence $\{A_nu_n^0\}$ is bounded
	 in $X$. Using this result, inequality (\ref{13}) implies that
	 the sequence $\{ \lambda_n^0 \}$ is bounded in $Y$, i.e., there exists $P_0>0$ which does not depend on $n$ such that
	 \begin{equation}\label{bla}
	 \| \lambda_n^0\|_Y\le P_0.
	 \end{equation}
	
	 Lemma \ref{l1} is now a direct consequence of the inequalities (\ref{bu}) and (\ref{bla}).
\end{proof}

\medskip
The second step is given by the following result.

\begin{lemma}\label{l2} For each $n\in\mathbb{N},$ Problem $\ref{pm}$ has a unique solution $(\wu_n,\wl_n)\in C(\real_+;X\times\Lambda_n)$. Moreover,
for	each $m\in \mathbb{N},$ there exists $\widetilde{a}_m>0$  such that
	\begin{equation}\label{x2}
	\| \wu_n(t)\|_X\le \widetilde{a}_m,\qquad\  \| \wl_n(t)\|_Y\le \widetilde{a}_m \qquad\forall\, t\in[0,m],\ n\in\mathbb{N}.
	\end{equation}
\end{lemma}

\begin{proof}  The existence and uniqueness part is a direct consequence of Proposition \ref{pr2}.
Let $m,\,n\in\mathbb{N}$ and let $t\in[0,m]$. Note that both Problems
\ref{pm} and \ref{pq} are problems of the form (\ref{7})--(\ref{8}).
Therefore, using Remark \ref{R1} it follows that we are
in the position to use Proposition \ref{pr1} to obtain the estimate
\begin{equation}\label{9p}
\|\wu_n(t)-u_n^0\|_{X}+\|\wl_n(t)-\lambda_n^0\|_{Y}\leq
d_0(\|\cS u(t)\|_{X}+\|f_n(t)\|_{Z}+\|h_n(t)\|_X)
\end{equation}
where $d_0>0$ is a positive contstant which does not depend on $n$. Denote
\[
\zeta_m=\max_{t\in[0,m]}\|Su(t)\|_X.
\]
Then, using inequality (\ref{9p}) and assumptions (\ref{cvvf})(b) and (\ref{cvvh})(b)
we deduce that there exists $K_m>0$ which does not depend on $n$ such that
\begin{equation}\label{9q}
\|\wu_n(t)-u_n^0\|_{X}+\|\wl_n(t)-\lambda_n^0\|_{Y}\leq
K_m.
\end{equation}
We now write
\begin{eqnarray*}
&&\|\wu_n(t)\|_{X}\leq
\|\wu_n(t)-u_n^0\|_{X}+\|u_n^0\|_{X},\\ [2mm]
&&\|\wl_n(t)\|_{Y}\le \|\wl_n(t)-\lambda_n^0\|_{Y}+\|\lambda_n^0\|_{Y}
\end{eqnarray*}
then we use  inequalities (\ref{9q}) and (\ref{x1}) to see that
(\ref{x2}) holds with $\widetilde{a}_m=K_m+a_0$.
\end{proof}

 The next step of the proof consists in the following convergence result.

 \begin{lemma}\label{l3}
 For each  $t\in\real_+,$ there exists a pair $(\wu(t),\wl(t))\in X\times Y$ and a subsequence of the sequence $\{(\wu_n,\wl_n)\}$, still denoted by
 $\{(\wu_n,\wl_n)\}$, such that
  \begin{eqnarray}
  &&\label{cus}\wu_n(t)\rightharpoonup \wu(t)\qquad{\rm in}\quad X,\\ [2mm]
  &&\label{cls}\wl_n(t)\rightharpoonup \wl(t)\qquad{\rm in}\quad Y.
  \end{eqnarray}
  Moreover,
  \begin{equation}
  \label{cusss}\wu_n(t)\to \wu(t)\qquad{\rm in}\quad X.
  \end{equation}
  \end{lemma}

 \begin{proof}  Let $t\in \real_+$ and let $m$ be such that $t\in[0,m]$.
 We use Lemma \ref{l2} and a standard compactness argument to see that there exists an element $\wu(t)\in X$
 and an element $\wl(t)\in Y$ such that (\ref{cus}) and (\ref{cls}) hold.

 We now prove the strong convergence (\ref{cusss}) and, to this end,
 we test with $v=\wu_n(t)-\wu(t)$ in (\ref{5n}) to deduce that
\begin{eqnarray*}
&&(A_n\wu_n(t),\wu_n(t)-\wu(t))_X+(\cS u(t),\wu_n(t)-\wu(t))_X\\ [2mm]
&&\qquad\quad+b_n(\wu_n(t)-\wu(t),\wl_n(t))=(f_n(t),\pi \wu_n(t)-\pi\wu(t))_Z.
\end{eqnarray*}
Therefore,
\begin{eqnarray*}
	&&(A_n\wu_n(t)-A_n\wu(t),\wu_n(t)-\wu(t))_X=(A_n\wu(t),\wu(t)-\wu_n(t))_X
	\\ [2mm]
	&&\quad+(\cS u(t),\wu(t)-\wu_n(t))_X+b_n(\wu(t)-\wu_n(t),\wl_n(t))\\ [2mm]
	&&\qquad+(f_n(t),\pi \wu_n(t)-\pi\wu(t))_Z
\end{eqnarray*}
and, using Remark \ref{R1} we find that
\begin{eqnarray}
	&&\label{zz}m_0\|\wu_n(t)-\wu(t)\|_X^2\le (A_n\wu(t),\wu(t)-\wu_n(t))_X
	\\ [2mm]
	&&\quad+(\cS u(t),\wu(t)-\wu_n(t))_X+b_n(\wu(t)-\wu_n(t),\wl_n(t))\nonumber\\ [2mm]
	&&\qquad+(f_n(t),\pi \wu_n(t)-\pi\wu(t))_Z.\nonumber
\end{eqnarray}

Next, using  (\ref{cvA})(a) we find that
 \begin{eqnarray*}
 	&&(A_n\widetilde{u}(t),\widetilde{u}(t)-\wu_n(t))_X=
 	(A_n\widetilde{u}(t)-A\widetilde{u}(t),\widetilde{u}(t)-\wu_n(t))_X+
 	(A\widetilde{u}(t),\widetilde{u}(t)-\wu_n(t))_X\\ [2mm]
 	&&\qquad \le \|A_n\widetilde{u}(t)-A\widetilde{u}(t)\|_X
 	\|\wu_n(t)-\widetilde{u}(t)\|_X+(A\widetilde{u}(t),\widetilde{u}(t)-u\wu_n(t))_X\\ [2mm]
 	&&\qquad\qquad \le F_n(\|\widetilde{u}(t)\|_X+\delta_n)\|\wu_n(t)-\widetilde{u}(t)\|_X+
 	(A\widetilde{u}(t),\widetilde{u}(t)-\wu_n(t))_X.
 \end{eqnarray*}
 We now pass to the upper limit in this inequality and use assumptions (\ref{cvA})(b), (c) and the convergence $\wu_n(t) \rightharpoonup {\widetilde{u}(t)}$ in $X$ to see that
 \begin{equation}\label{17}
 \limsup\, (A_n\widetilde{u}(t),\widetilde{u}(t)-\wu_n(t))_X\le 0.
 \end{equation}

 On the other hand, the convergence  (\ref{cus}) implies that
 \begin{equation}\label{19}
 (\cS u(t),\wu(t)-\wu_n(t))_X\to 0.
 \end{equation}
 Moreover, taking  $z_n=\wu_n(t)$, $\mu_n=\wl_n(t)$ and $w=\widetilde{u}(t)$ in  (\ref{cvb}) yields
 \begin{equation}\label{18}
 \limsup\, b_n(\widetilde{u}(t)-\wu_n(t),\wl_n(t))\le 0.
 \end{equation}
 In addition, using assumptions (\ref{cvvf})(a), (\ref{cvpi})
 we find that
 \begin{equation}\label{20}
 (f_n(t),\pi \wu_n(t)-\pi\wu(t))_Z\to 0.
 \end{equation}

 We now pass to the upper limit in the inequality (\ref{zz}) and use (\ref{17})--(\ref{20}) to deduce that
 \[\limsup\,m_0\|\wu_n(t)-\widetilde{u}(t)\|^2_X\le 0\]
 which shows that (\ref{cusss}) holds  and concludes the proof.
 \end{proof}

 \medskip
 The next step completes the statement of Lemma  \ref{l3} and it is as follows.

\begin{lemma}\label{l4}	For each $t\in\real_+,$ the following convergences hold:
	 \begin{eqnarray}
	 &&\label{cuss}\wu_n(t)\to u(t)\qquad{\rm in}\quad X,\\ [2mm]
	 &&\label{clss}\wl_n(t)\rightharpoonup \lambda(t)\qquad{\rm in}\quad Y.
	 \end{eqnarray}
	\end{lemma}

 \medskip\noindent{\it Proof.} Let $t\in\real_+$ and let $m\in\mathbb{N}$ such that $t\in[0,m]$. We first recall that for each $n\in\mathbb{N}$ we have $\wl_n(t)\in\Lambda_n$. Moreover,  using Lemma \ref{l3} it follows that passing to a subsequence, still denoted by $\{\wl_n(t)\}$, the convergence (\ref{cls}) holds.
 Therefore, assumption (\ref{cvLa})(b) implies that
 \begin{equation}
 \label{190}\widetilde{\lambda}(t)\in\Lambda.
 \end{equation}

 Next, we consider a subsequence of the sequence
 $\{(\wu_n,\wl_n)\}$, still denoted by
 $\{\wu_n,\wl_n)\}$, such that (\ref{cus}) and (\ref{cls}) hold.
 Let $n\in\mathbb{N}$ and $v\in X$. We use assumptions (\ref{cvA})(a) and (\ref{A})(b) to see that
 \begin{eqnarray*}
 	&&\|A_n\wu_n(t)-A\widetilde{u}(t)\|_X\le \|A_n\wu_n(t)-A\wu_n(t)\|_X+\|A\wu_n(t)-A\widetilde{u}(t)\|_X\\ [2mm]
 	&&\quad\le
 	F_n(\|\wu_n(t)\|_X+\delta_n)+ L_A\,\|\wu_n(t)-\widetilde{u}(t)\|_X
 \end{eqnarray*}
 and, therefore, assumptions (\ref{cvA})(b),(c) and (\ref{cusss}) imply that
 \begin{equation}
 \label{200}A_n\wu_n(t)\to A\widetilde{u}(t)\quad{\rm in}\quad X.
 \end{equation}
 On the other hand, we write condition (\ref{cvb}) with
 $z_n=0_X$, $\mu_n=\wl_n(t)$, $w=v$, then with $z_n=v$, $\mu_n=\wl_n(t)$ and $w=0_X$ to obtain
 \[\limsup\,b_n(v,\wl_n(t))\le b(v,\widetilde{\lambda}(t))
 \quad{\rm and}\quad
 b(v,\widetilde{\lambda}(t))\le\liminf\,b_n(v,\wl_n(t)), \]
 respectively. These inequalities show that
 \begin{equation}
 \label{21}b_n(v,\wl_n(t))\to b(v,\widetilde{\lambda}(t)).
 \end{equation}
 Finally, note that the convergence (\ref{cvvf})(a) implies that
 \begin{equation}
 \label{22}(f_n(t),\pi v)_Z\to (f(t),\pi v)_Z.
 \end{equation}

 Next, we pass to the limit in  equality (\ref{5n}) and use the convergences (\ref{200})--(\ref{22}) to see that
 \begin{equation}
 \label{23}
 (A\widetilde{u}(t),v)_X+(\cS u(t),v)_X+b(v,\widetilde{\lambda}(t))=(f(t),\pi v)_Z.
 \end{equation}

 Consider now an arbitrary element $\mu\in\Lambda$. Using assumption (\ref{cvLa})(a) we know that there exists a sequence
 $\{\mu_n\}$ such that
 $\mu_n\in \Lambda_n$ for each $n\in \mathbb{N}$  and $\mu_n\to \mu$ in $Y$. This allows to use the inequality (\ref{6n}) to see that
 \[b_n(\wu_n(t)-h_n(t),\mu_n-\wl_n(t))\le 0\]
 which implies that
 \begin{equation}\label{24}
 \liminf\,b_n(\wu_n(t)-h_n(t),\mu_n-\wl_n(t))\le 0.
 \end{equation}
 On the other hand, writing condition (\ref{cvb}) with $w=0_X,z_n=\wu_n(t)-h_n(t)$ and $\mu_n-\wl_n(t)$ instead of $\mu_n,$ we deduce that
 \begin{equation*}
 \limsup\,b_n(-\wu_n(t)+h_n(t),\mu_n-\wl_n(t))\le b(-\widetilde{u}(t)+h(t),\mu-\widetilde{\lambda}(t))
 \end{equation*}or, equivalently,
 \begin{equation}\label{25}
 b(\widetilde{u}(t)-h(t),\mu-\widetilde{\lambda}(t))\le \liminf\,b_n(\wu_n(t)-h_n(t),\mu_n-\wl_n(t)).
 \end{equation}
 We  combine inequalities (\ref{25}) and (\ref{24}) to find that
 \begin{equation}\label{26}
 b(\widetilde{u}(t)-h(t),\mu-\widetilde{\lambda}(t))\le 0.
 \end{equation}

 Finally,  we gather  (\ref{190}), (\ref{23}) and (\ref{26}) to  conclude that the pair $(\wu(t),\wl(t))$ satisfies (\ref{1})--(\ref{2}). On the other hand, it follows from Proposition \ref{pr1}  that there exists a unique solution to the system (\ref{1})--(\ref{2}),  denoted $(u(t),\lambda(t))$. Therefore,  we deduce that ${\widetilde{u}(t)} = u(t)$ and $\widetilde{\lambda}(t)=\lambda(t)$.
 Next, the proofs of Lemmas \ref{l2}--\ref{l3} combined with equalities ${\widetilde{u}(t)} = u(t)$ and $\widetilde{\lambda}(t)=\lambda(t)$ reveal
 the fact that the sequence $\{(\wu_n(t),\wl_n(t))\}$ is bounded in $X\times Y$
 and every  subsequence of $\{(\wu_n(t),\wl_n(t))\}$
 which converges weakly in $X\times Y$ has the same limit $(u(t),\lambda(t))$. Therefore, by a standard argument we deduce  that the whole sequence $\{(\wu_n(t),\wl_n(t)\}$ converges weakly in $X\times Y$ to $(u(t),\lambda(t))$ or, equivalently,  $\wu_n(t) \rightharpoonup {u}(t)$ in $X$ and
 $\wl_n(t)\rightharpoonup {\lambda}(t)$  in $Y$.
 This shows  that  the weak convergence $(\ref{clss})$ holds. Moreover, repeating the arguments in Lemma \ref{l2} we deduce  the strong convergence  $(\ref{cuss})$, which
 concludes the proof.
 \hfill$\Box$

  \medskip We are now in a position to provide the proof of Theorem \ref{t1}.

 \medskip\noindent	
 \begin{proof}
 Let $n\in\mathbb{N}$, $t\in\real_+$ and let $m\in\mathbb{N}$ be such that $t\in[0,m]$. A careful examination of Problems \ref{pn} and \ref{pq} reveals that, with an appropriate notation,
 these problems are governed by a system of the  (\ref{7})--(\ref{8}). Therefore, using Remark \ref{R1}
  we are allowed to apply Proposition \ref{pr1} to obtain the  estimate
\begin{equation}\label{29}
\|u_n(t)-\wu_n(t)\|_{X}+\|\lambda_n(t)-\wl_n(t)\|_{Y}\leq
d_0\|\cS_nu_n(t)-\cS u(t)\|_{X}
\end{equation}
where $d_0$ is a positive constant which depends only on $m_0$, $L_0$, $\alpha_0$, $M_0$ and $c_0$.
Thus,
\begin{equation}\label{30}
\|u_n(t)-\wu_n(t)\|_{X}\leq
d_0\|\cS_nu_n(t)-\cS u(t)\|_{X}.
\end{equation}

Let
\begin{equation}\label{30n}
r_m=\max_{t\in[0,m]}\|u(t)\|_X.
\end{equation}
We  write\  $\cS_nu_n(t)-\cS u(t)=\cS_nu_n(t)-\cS_n u(t)+\cS_nu(t)-\cS u(t)$,
then we use assumption (\ref{cvS}) and Remark \ref{R1}, again,
to deduce that
\begin{equation}\label{31}
\|\cS_nu_n(t)-\cS u(t)\|_{X}\le s_m^0\int_0^t\|u_n(s)-u(s)\|_{X}\,ds+F_n^m(r_m+\delta_n^m).
\end{equation}

On the other hand, we have
\begin{equation}\label{32}
\|u_n(t)-u(t)\|_{X}\leq \|u_n(t)-\wu_n(t)\|_{X}+\|\wu_n(t)-u(t)\|_{X}.
\end{equation}
 Therefore, combining (\ref{32}), (\ref{30}) and  (\ref{31})  we find that
 \begin{eqnarray*}
&& \|u_n(t)-u(t)\|_{X}\leq \|\wu_n(t)-u(t)\|_{X}\\ [2mm]
 &&\qquad+d_0s_m^0\int_0^t\|u_n(s)-u(s)\|_{X}\,ds+d_0F_n^m(r_m+\delta_n^m)\nonumber
 \end{eqnarray*}
 and, using the Gronwall argument yields
 \begin{eqnarray}
 &&\label{33}
 \|{u}_n(t)-u(t)\|_X\le d_0F_n^m(r_m+\delta_n^m)+\|\wu_n(t)-u(t)\|_{X}\\ [2mm]
  &&\quad
 +d_0s_m^0\int_0^te^{d_0s_m^0(t-s)}\Big(d_0F_n^m(r_m+\delta_n^m)+\|\wu_n(s)-u(s)\|_{X}\Big)\,ds.\nonumber
 \end{eqnarray}

 For all $s\in[0,m]$ denote
 \begin{equation}\label{35}
 z_n^m(s)=d_0F_n^m(r_m+\delta_n^m)+\|\wu_n(s)-u(s)\|_{X}.
 \end{equation}
 Then, (\ref{33}) implies that
 \begin{eqnarray}
 &&\label{38}
 \|{u}_n(t)-u(t)\|_X\le z_n^m(t)+d_0s_m^0\int_0^te^{d_0s_m^0(t-s)}z_n^m(s)\,ds.
 \end{eqnarray}
 On the other hand,  (\ref{x2}) and (\ref{30n}) imply that
\[
\|\wu_n(s)-{u}(s)\|_X\le \|\wu_n(s)\|_X+\|{u}(s)\|_X\le
\widetilde{a}_m+r_m
\]
and, therefore, (\ref{35}) yields
\begin{equation}\label{util}
|z_n^m(s)|\le d_0F_n^m(r_m+\delta_n^m)+\widetilde{a}_m+r_m\qquad\forall\, s\in[0,m].
\end{equation}
We now use assumption (\ref{cvS})(b), (c)  and Lemma \ref{l4} to see that the sequence of functions $\{z_n^m\}$ is bounded as $n\to\infty$
and converges  to zero, for all $s\in[0,m]$, i.e.,
\begin{eqnarray}
&&\hspace{-6mm}\mbox{there exists\ $Z_m>0$\ such that}\ \  | z_n^m(s)|_X\le Z_m\quad\forall s\in[0,m],\ n\in\mathbb{N}, \label{zz1}\\ [2mm]
&&\hspace{-6mm}z_n^m(s)\to 0\quad{\rm as}\ \ n\to\infty,\ \ \ \forall\,s\in[0,m].\label{zz2}
\end{eqnarray}
Therefore, we are in a position  to use the Lebesgue theorem in order to  see that
 \begin{eqnarray}
 &&\label{r1n}
 \int_0^te^{d_0s_m^0(t-s)}z_n^m(s)\,ds\to 0,
 \end{eqnarray}
 and, moreover,
 \begin{eqnarray}
 &&\label{r2}z_n^m(t)\to 0\qquad{\rm as}\ n\to\infty.
 \end{eqnarray}
 We now use the convergences (\ref{r1n}), (\ref{r2}) and inequality (\ref{38}) to deduce that (\ref{cvu}) holds.

 Finally, note that (\ref{38}), (\ref{util}) and the convergence (\ref{cvu})  allows us to use the Lebesgue theorem, again, in order to see that
 \begin{eqnarray*}\label{33p}
 \int_0^t\|u_n(s)-u(s)\|_{X}\,ds\to 0
 \end{eqnarray*}
 and, in addition,  assumption (\ref{cvS})(b), (c)  yield
 \begin{eqnarray*}
 F_n^m(r_m+\delta_n^m)\to 0.
 \end{eqnarray*}
 These two convergences combined with (\ref{31}) imply that
 \begin{equation*}\label{31n}
 \|\cS_nu_n(t)-\cS u(t)\|_{X}\to 0
 \end{equation*}
 and, using inequality (\ref{29}) we deduce that
 \begin{equation}\label{40}
 \|\lambda_n(t)-\wl_n(t)\|_{Y}\to0.
 \end{equation}
 We now write $\lambda_n(t)-\lambda(t)=\lambda_n(t)-\wl_n(t)+\wl_n(t)-\lambda(t)$
 then we use the strong convergence (\ref{40}) and the weak convergence (\ref{clss}) to find that
  (\ref{cvla}) holds, which concludes the proof.
 \end{proof}

 Note that Theorem \ref{t1} states a pointwise convergence result, strongly for the first component of the solution $({u,\lambda})$ of Problem \ref{p}, and weakly for the second component. Considering appropriate assumptions on the data which guarantee a strong convergence result for the second component $\lambda$ and/or a uniform convergence result for the solution $({u,\lambda})$ represents an open problem which, clearly, deserves to be studied in the future.

\section{An optimization problem}\label{s4}

In this section we apply Theorem \ref{t1} in the study of a general optimization problem associated to the history-dependent mixed variational problem problem (\ref{1})--(\ref{2}). To this end we consider a reflexive Banach space $W$  endowed with the norm $\|\cdot\|_W$ and a nonempty subset $U\subset W$.
For each $p\in U$ we consider two operators  $A_p$, $\cS_p$, a form $b_p$ and a set $\Lambda_p$  which satisfy assumptions
(\ref{A}),  (\ref{S}), (\ref{b}), (\ref{lam}), respectively, with constants $m_{p}$, $L_p$, $s_m^p$, $M_p$, $\alpha_p$. To avoid any confusion, when used with $p$, we refer to these  assumptions as assumptions $(\ref{A})_p$, $(\ref{S})_p$, $(\ref{b})_p$, $(\ref{lam})_p$. Also, assume that the elements $\widetilde{f}_p$ and
$\widetilde{h}_p$ are given and have the regularity
\begin{eqnarray}
&&\label{51}\wf_p\in Z, \\ [2mm]
&&\label{52}\wh_p\in X.
\end{eqnarray}
Let $\theta$ and $\zeta$ be two functions such that
\begin{eqnarray}
&&\label{53}\theta\in C(\real_+;\real), \\ [2mm]
&&\label{54}\zeta\in C(\real_+;\real)
\end{eqnarray}
and consider the functions $f_p$, $h_p$ defined by
\begin{eqnarray}
&&\label{55}f_p:\real_+\to Z,\qquad f_p(t)=\theta(t)\wf_p\qquad\forall\, t\in\real_+,\\ [2mm]
&&\label{56}h_p:\real_+\to Z,\qquad f_p(t)=\zeta(t)\wh_p\qquad\forall\, t\in\real_+.
\end{eqnarray}
Then, under the previous assumptions, if in addition the condition (\ref{pi}) is satisfied, we deduce from Proposition \ref{pr2} that for each $p\in U$ there exists a unique solution $(u_p,\lambda_p)\in C(\real_+;X\times\Lambda_p)$ to the following problem.

\begin{problem}\label{ppp} Find $u_p:\real_+\to X$ and $\lambda_p:\real_+\to \Lambda_p$
	such that
\begin{eqnarray}
&&\hspace{-5mm}(A_pu_p(t),v)_X+(\cS_p u_p(t),v)_X+b_p(v,\lambda_p(t))=(f_p(t),\pi v)_Z
\qquad \forall\, v\in X,\label{1p}\\[2mm]
&&\hspace{-5mm}\qquad b(u_p(t),\mu-\lambda_p(t))\leq
b_p(h_p(t),\mu-\lambda_p(t)) \qquad\forall\, \mu\in
\Lambda_p\label{2p}
\end{eqnarray}
for all $t\in \mathbb R_+$.	
\end{problem}

Consider
a cost function ${\cal L}:X\times Y\times U \to\mathbb{R}$. We formulate the following optimization problem.

\medskip\noindent

\begin{problem}\label{o}   Given $t\in\real_+$,  find  $p^*\in U$ such that
	\begin{equation}\label{o1}
	{\cal L}(u_{p^*}(t),\lambda_{p^*}(t),p^*)=\min_{p\in U} {\cal L}(u_p(t),\lambda_p(t),p).
	\end{equation}
\end{problem}

\medskip
To solve Problem \ref{o} we consider the following assumptions.
\begin{eqnarray}
&&\label{o2} \qquad U\  \ \mbox{is a nonempty weakly closed subset of}\ \  W.\qquad
\\[4mm]
&&\label{o3}\left\{\begin{array}{l}
\mbox{For all sequences }\{u_n\}\subset X, \{\lambda_n\}\subset Y \mbox{ and
}\{p_n\}\subset
U \mbox{ such that}\\[2mm]
u_n\rightarrow u\mbox{\ \ in\ \ }X, \ \ \lambda_{n}\rightharpoonup
\lambda \mbox{\ \ in\ \ }
Y,\ \ p_{n}\rightharpoonup
p \mbox{\ \ in\ \ }
W,\ \mbox{we have} \\[3mm]
\displaystyle\liminf\,{\cal L}(u_n,\lambda_n,p_n)\ge {\cal L}(u,\lambda,p).
\\[3mm]
\end{array}\right.
\\[4mm]
&&\label{o4}\left\{ \begin{array}{l}
\mbox{ There exists}\ z: U\to\R\ \mbox{such that}
\\[2mm]
{\rm (a)}\ \ {\cal L}(u,\lambda,p)\ge z(p)\quad \forall\,u\in X,\  \lambda\in Y,\ p\in U,\\ [2mm]
{\rm (b)}\ \ \|p_{n}\|_{W}\to+\infty\ \Longrightarrow\ z(p_n)\to \infty.
\end{array}\right.\\[4mm]
&&\label{o5}\qquad U \ \ \mbox {is a bounded subset of} \ \  W.
\end{eqnarray}

	A typical example of function ${\cal L}$ which satisfies conditions $(\ref{o3})$ and $(\ref{o4})$ is obtained by taking
	\[
	{\cal L}(u,\lambda,p)=g(u)+k(\lambda)+z(p)\qquad\forall\, u\in X, \ \lambda\in Y, \ p\in U,
	\]
	where $g:X\to\R_+$ is a lower semicontinuous function, $k:Y\to\R_+$ is a weakly lower semicontinuous function, and $z: U\to \R$ is a weakly lower semicontinuous coercive function, i.e., it satisfies condition $(\ref{o4}){\rm (b)}$.		

\medskip
Our main result in this section is the following.

\medskip
\begin{theorem}\label{t2}  Assume
	$(\ref{A})_p$--$(\ref{lam})_p$ and  $(\ref{51})$--$(\ref{56})$ for each $p\in U$.
	Moreover, assume $(\ref{pi})$, $(\ref{o2})$, $(\ref{o3})$  and either
	$(\ref{o4})$ or $(\ref{o5})$. In addition, assume that
	for each sequence $\{p_n\}\subset U$  such that $p_{n}\rightharpoonup p$ in $W$,
	conditions $(\ref{cvA})$--$(\ref{cvLa})$ are satisfied
	with $A_n=A_{p_n}$, $A=A_p$, $m_n=m_{A_{p_n}}$, $L_n=L_{A_{p_n}}$, $\cS_n=\cS_{p_n}$, $\cS=\cS_p$, $s_m^n=s_m^{p_n}$, $s_m^p=s_m$, $b_{p_n}=b_n$, $\alpha_n=\alpha_{p_n}$, $\alpha=\alpha_p$, $\Lambda_n=\Lambda_{p_n}$, $\Lambda=\Lambda_p$ and
	\begin{eqnarray}
&&\label{57}\wf_{p_n}\rightharpoonup \wf_p \quad\ \ \, {\rm in}\ Z.\\ [2mm]
&&\label{58}\wh_{p_n}\rightharpoonup \wh_p\quad \quad{\rm in}\ X.
\end{eqnarray}
Then, for each $t\in\mathbb{R}_+$,  Problem $\ref{o}$ has  at least one solution $p^*$.
\end{theorem}

\begin{proof} Let $t\in\mathbb{R}_+$ be fixed.
We  consider the function $J_t:U\to\R$ defined by
\begin{equation}\label{Jmn}
J_t(p)=	{\cal L}(u_p(t),\lambda_p(t),p)\qquad\forall\,p\in U
\end{equation}
together with the problem of finding  $p^*\in U$ such that
\begin{equation}\label{o1p}
J_t(p^*)=\min_{p\in U} J_t(p).
\end{equation}

Assume that $\{p_n\}\subset U$ is such that $p_n\rightharpoonup p\mbox{ in }W$.  Then, recall that
(\ref{57}) and (\ref{58}) hold. Note that
$(\ref{51})$, $(\ref{53})$, $(\ref{55})$ and (\ref{57}) show that the functions $f_n=f_{p_n}$ and $f=f_p$
satisfy condition (\ref{cvvf}). Moreover, $(\ref{52})$, $(\ref{54})$, $(\ref{56})$ and (\ref{58}) show that the functions $h_n=h_{p_n}$ and $h=h_p$
satisfy condition (\ref{cvvh}). Thus,  since conditions  $(\ref{cvA})$--$(\ref{cvLa})$,
are satisfied (in the sense prescribed in the statement of the theorem), we are in a position to apply Theorem
\ref{t1} in order to obtain that
$u_{p_n}(t)\to u_p(t)$ in $X$ and  $\lambda_{p_n}(t)\rightharpoonup \lambda_p(t)$ in $Y$. Therefore, using definition (\ref{Jmn})
and assumption (\ref{o3})  we deduce that
\begin{equation*}\label{r1}
\liminf J_t(p_n)=\liminf {\cal L}(u_{p_n}(t),\lambda_{p_n}(t),p_n)\ge {\cal L}(u_{p}(t),\lambda_{p}(t),p)=J_t(p).
\end{equation*}
It follows from here that the function $J_t:U\to\real$  is weakly lower semicontinuous.

Assume now that  (\ref{o4}) holds. Then,
for any sequence $\{p_n\}\subset U$,
we have
\[J_t(p_n)={\cal L}(u_{p_n}(t),\lambda_{p_n}(t),p_n)\ge z(p_n).\]
Therefore, if $\|p_n\|_{W} \to \infty$ we deduce that
$J_t(p_n)\to \infty$ which shows that
$J_t:U\to\real$ is coercive.
Recall also the assumption (\ref{o2}) and the reflexivity of the space $W$. The existence of at least one solution to problem (\ref{o1p}) is now a direct consequence of Theorem \ref{tW}. On the other hand, if we assume that condition (\ref{o5}) is satisfied we are still in a position to apply  Theorem \ref{tW}. We deduce from here that, if either (\ref{o4}) or (\ref{o5})
holds, then there exists at least one solution $p^*\in U$ to the optimization problem (\ref{o1p}).
We now use the definition (\ref{Jmn}) to see that $p^*$ is a solution to Problem \ref{o} which concludes the proof.
\end{proof}

\section{A viscoelastic frictional contact problem}\label{s5}

The abstract results  in Sections \ref{s3}--\ref{s4} are useful in the variational analysis of  mathematical models which describe the equilibrium of deformable bodies in  contact with an obstacle, the so-called foundation. In this section
we illustrate their use  in the study of a  frictional contact model with linearly viscoelastic materials. For the description of additional  models of contact as well as for details on the  notation and preliminaries introduced below   we refer the reader
to  \cite{HS,SST,SM,SMBOOK}.

The physical setting is as follows. We consider a viscoelastic body which occupies a
bounded domain $\Omega\subset\mathbb{R}^d$ ($d=2,3$) with a
Lipschitz continuous boundary $\Gamma$, divided into three
measurable disjoint parts $\Gamma_1$, $\Gamma_2$ and $\Gamma_3$ such that
${ meas}\,(\Gamma_1)>0$.  We assume that the body is fixed on $\Gamma_1$, is acted by given body forces and given surface tractions on $\Gamma_2$, and it is in frictional contact
with an obstacle on $\Gamma_3$.  The time interval of interest is $\real_+=[0,+\infty)$, the contact is bilateral, that is, there is no separation between the body and the foundation, and it is associated to the
Tresca friction law. Then,
the equilibrium of the  body in the physical setting above is described by the following boundary value problem.

\medskip\noindent
\begin{problem}
	\label{Pm}
	{\it Find a displacement field
		$\bu:\Omega\times\real_+\to\mathbb{R}^d$ and a stress field
		$\bsigma:\Omega\times\real_+\to\mathbb{S}^d$ such that}
	\begin{eqnarray}
	\hspace{-16mm}\label{1m} \bsigma(t)=2\beta\bvarepsilon({\bu}(t))+\eta\, {\rm tr}(\bvarepsilon({\bu}(t))I_d+\int_0^te^{-\omega(t-s)}\bvarepsilon({\bu}(s))\,ds\ &{\rm in}\
	&\Omega,\\[3mm]
	\label{2m} {\rm Div}\,\bsigma(t)+\fb_0(t)=\bzero\quad&{\rm
		in}\ &\Omega,\\[3mm]
	\label{3m} \bu(t)=\bzero\quad &{\rm on}\ &\Gamma_1,\\[3mm]
	\label{4m} \bsigma\bnu(t)=\fb_2(t)\quad&{\rm on}\
	&\Gamma_2,\\[3mm]
	\hspace{-8mm}\label{5mm} u_\nu(t)=0,\quad
	\|{\bsigma}_\tau(t)\|\leq g,\quad {\bsigma}_\tau(t)=-g\,\frac{\bu_{\tau}(t)}{\|\bu_{\tau}(t)\|}\ \mbox{ \rm if }\ \bu_{\tau}(t)\neq \bzero\, \ &{\rm on}\ &\Gamma_3,
	\end{eqnarray}
	for all $t\in\real_+$.
\end{problem}

\medskip
Here and below in this section we do not mention the dependence of various functions  with respect to the spatial variable $\bx\in\Omega\cup\Gamma$.  Notation $\mathbb{S}^d$ represents the space of second order symmetric
tensors on $\mathbb{R}^d$ or, equivalently, the space of symmetric
matrices of order $d$, and $I_d$ stands for the unit tensor of $\mathbb{S}^d$. The  inner product and norm on
$\mathbb{R}^d$ and $\mathbb{S}^d$ are defined by
\begin{eqnarray*}
	&&\bu\cdot \bv=u_i v_i\ ,\qquad\ \
	\|\bv\|=(\bv\cdot\bv)^{\frac{1}{2}}\qquad
	\forall \,\bu, \bv\in \mathbb{R}^d,\\[0mm]
	&&\bsigma\cdot \btau=\sigma_{ij}\tau_{ij}\ ,\qquad
	\|\btau\|=(\btau\cdot\btau)^{\frac{1}{2}} \qquad \,\forall\,
	\bsigma,\btau\in\mathbb{S}^d,
\end{eqnarray*}
\noindent  and the zero element of these spaces will be denoted by $\bzero$. Also,  ${\bnu}$ is the outward unit normal at $\Gamma$ and
$u_\nu$, $\bu_\tau$  represent the normal and
tangential components of $\bu$ on $\Gamma$ given by
$u_\nu=\bu\cdot\bnu$ and $\bu_\tau=\bu-u_\nu\bnu$, respectively. Finally, $\sigma_\nu$ and $\bsigma_\tau$ denote  the
normal and tangential stress on $\Gamma$, that is
$\sigma_{\nu}=(\bsigma\bnu)\cdot\bnu$ and $\bsigma_{\tau} =
\bsigma\bnu - \sigma_{\nu}\bnu$.

We now provide a short description of the equations and boundary conditions in Problem
$\ref{pm}$. First, equation (\ref{1m}) represents the viscoelastic constitutive law of the material, in which $\beta$ and $\eta$ represent the Lam\'e coefficients, ${\rm tr\,\btau}$ denotes the trace of the tensor $\btau$,  $\omega$ is a relaxation coefficient and $\bvarepsilon(\bu)$ denotes the linearized strain field.
Equation (\ref{2m}) is
the equation of equilibrium in which $\fb_0$ represents the density of the body forces and ${\rm Div}$ denotes the divergence operator.
Conditions (\ref{3m}), (\ref{4m}) represent the
displacement and traction boundary conditions, respectively, where $\fb_2$ denotes the density of given surface tractions which act on the part $\Gamma_2$ of the boundary. Finally, condition (\ref{5mm}) represents the interface law on the contact surface. Equality $u_\nu(t)=0$ shows that there is no separation between the body and the obstacle during the deformation process, i.e., the contact is bilateral. The rest of the condition in (\ref{5mm}) represent the static version of the Tresca's friction law, in which $g$ denote the friction bound, assumed to be given.

In the study of the contact problem (\ref{1m})--(\ref{5mm}) we
assume that the data satisfy the following conditions.

\begin{eqnarray}
&&\label{be} \beta\ge0,\\
&&\label{et} \eta\ge0,\\
&&\label{om} \omega\ge0,\\
&&\label{f0} \fb_0\in C(\real_+;L^2(\Omega)^d),\\
&&\label{f2}\fb_2\in C(\real_+;L^2(\Gamma_2)^d), \\
&&\label{gm}g\ge0.
\end{eqnarray}

Everywhere below
we use the standard notation for Sobolev and Lebesgue spaces associated to
$\Omega$ and $\Gamma$ and we denote by $\gamma:H^1(\Omega)^d\to L^2(\Gamma)^d$ the trace operator. For each element $\bv\in H^1(\Omega)^d$ we use the notation $v_\nu$ and $\bv_\tau$ for the  normal and tangential components of $\bv$ on $\Gamma$, that is,
$v_\nu=\gamma\bv\cdot\bnu$ and $\bv_\tau=\gamma\bv-v_\nu\bnu$, respectively.
Moreover, we use the notation $\bvarepsilon(\bv)$
for the associated linearized strain field, i.e.,
\[
\bvarepsilon(\bv)=(\varepsilon_{ij}(\bv)),\quad
\varepsilon_{ij}(\bv)=\frac{1}{2}\,(v_{i,j}+v_{j,i}),
\]
where an index that follows a comma denotes the
partial derivative with respect to the corresponding component of $\bx$, e.g.,\ $v_{i,j}=\frac{\partial v_i}{\partial x_j}$.

Next, for the displacement field we consider the
space
\begin{eqnarray*}
	&&X=\{\,\bv\in H^1(\Omega)^d:\  \gamma\bv =\bzero\ \ {\rm on\ \ }\Gamma_1,\quad v_\nu=0\ \ {\rm on\ \ }\Gamma_3\,\}.
\end{eqnarray*}
Since ${ meas}\,(\Gamma_1)>0$, it is well known that  $X$
is a real Hilbert space
endowed with the canonical inner product
\begin{equation}
(\bu,\bv)_X= \int_{\Omega}
\bvarepsilon(\bu)\cdot\bvarepsilon(\bv)\,dx
\end{equation}
and the associated norm
$\|\cdot\|_X$.  It follows from
\cite{ANZIAM} that the space $\gamma(X)$ is a closed subspace of the Hilbert space $\gamma(H^1(\Omega)^d)$ and, therefore, it  is
a Hilbert space, too.
Let $Y$ be its dual (which, in turn,
can be organized as a real Hilbert space) and denote by $\langle\cdot,\cdot\rangle$ the
duality pairing
between $Y$ and $\gamma(X)$.  Recall also that $\gamma(X)$ is continuously embedded in $L^2(\Gamma)^d$. Finally, we need the space $Z=L^2(\Omega)^d\times L^2(\Gamma_2)^d$ equipped with the canonical inner product.

We now introduce
the operators $A:X\to X$, $\cS:C(\real_+,X)\to C(\real_+,X)$ and $\pi:X\to Z$, the form $b:X\times Y \to\real$, the set $\Lambda\subset Y$ and  the function $\fb:\real_+\to Z$ defined by the following equalities:
\begin{eqnarray}
&&\label{Am} (A\bu,\bv)_X=\int_\Omega \Big(2\beta\bvarepsilon({\bu})+\eta\, {\rm tr}(\bvarepsilon({\bu}))I_d\Big)\cdot\bvarepsilon(\bv)\,dx\quad\forall\, \bu,\, \bv\in
X,\\[4mm]
&&\label{Sm} (\cS\bu(t),\bv)_X=\int_\Omega\Big(\int_0^te^{-\omega(t-s)}\bvarepsilon({\bu}(s))\,ds\Big)\cdot\bvarepsilon(\bv)\,dx\\[4mm]
&&\qquad\qquad\quad\qquad\qquad\qquad\ \forall\, \bu\in C(\real_+;X),\ \bv\in
X, \nonumber
\end{eqnarray}
\begin{eqnarray}
&&\label{bm}b(\bv,\,\bmu)=\langle\bmu,
\gamma\bv\rangle,\qquad \forall\, \bv\in X,\ \mu\in Y,\\[4mm]
&&\label{pim}\pi \bv=(\bv,\gamma_2\bv)\qquad\forall\, \bv\in X, \\[4mm]
&&\label{Lam}\Lambda=\Big\{\bmu\in
Y\,:\,\langle\bmu,\,\bxi\rangle\leq g
\int_{\Gamma_3}\|\bxi\|da\quad \forall\, \bxi\in
\gamma(X)\Big\},\\[4mm]
&&\label{ff}\fb(t)=(\fb_0(t),\fb_2(t))\qquad\forall\, t\in\real_+.
\end{eqnarray}

Note that the definition of the operators $A$ and
$\cS$ follows by using Riesz's representation theorem. Moreover,
here and below, $\gamma_2\bv\in L^2(\Gamma_2)^d$ denotes the restriction to $\Gamma_2$ of the trace $\gamma\bv\in L^2(\Gamma)^d$, for any $\bv\in X$.
In addition, the definitions  (\ref{pim}) and
 (\ref{ff}) imply that
\begin{equation}
\label{ffm}(\fb(t),\pi\bv)_Z =\int_{\Omega}\fb_0(t)\cdot\bv\,dx
+\int_{\Gamma_2}\fb_2(t)\cdot\gamma_2\bv\,da\qquad\forall\,\bv\in X,\ t\in\real_+.
\end{equation}

We now introduce a new variable, the Lagrange multiplier, denoted by $\blambda$. It is related to the friction force $\bsigma_{\tau}$ on the contact zone $\Gamma_3$ by equality
\begin{equation}\label{lag}
\langle\blambda(t),\widetilde{\bv}\rangle=-\int_{\Gamma_3}\bsigma_{\tau}(t)\cdot \widetilde{\bv}\,da \qquad \forall\, \widetilde{\bv}\in \gamma(X),\ t\in\real_+.
\end{equation}

A mixed variational formulation of Problem \ref{Pm} can be easily obtained, based
on  equalities (\ref{ffm}), (\ref{lag}) and integration by parts. It can be stated as follows.

\medskip

\begin{problem}\label{pv}
	Find a displacement field $\bu:\real_+\to X$ and a Lagrange multiplier $\blambda:\real_+\to\Lambda$ such that
	\begin{eqnarray}
	&&\label{m1}(A\bu(t),\bv)_{X}+(\cS\bu(t),\bv)_{X}+b(\bv,\blambda(t))=(\fb(t),\,\pi\bv)_Z \quad\forall\, \bv\in X,\\ [2mm]
	&&\label{m2}b(\bu(t),\bmu-\blambda(t))\le 0\quad       \forall\,\bmu\in
	\Lambda
	\end{eqnarray}
	for all $t\in\real_+$.
\end{problem}

\medskip
The unique solvability of Problem \ref{pv} is given by the following existence and uniqueness result.

\begin{theorem}\label{t6}
	Assume $(\ref{be})$--$(\ref{gm})$.
	Then, Problem $\ref{pv}$ has a unique solution $(\bu,\blambda)\in C(\real_+;X\times \Lambda)$.
\end{theorem}

\noindent
\begin{proof} Let $\bu$, $\bv$, $\bw\in X$.
We use definition (\ref{Am}), inequality
\[{\rm tr}\,(\btau) I_d\cdot \btau=\big({\rm tr}\,(\btau)\big)^2\geq 0\qquad\forall\,\btau\in\mathbb{S}^d\]
and  assumption (\ref{et}) to see that
\begin{equation}
\label{A1}(A\bu-A\bv,\bu-\bv)_{X}
\geq 2\beta\|\bu-\bv\|_X^2.
\end{equation}
On the other  hand,
assumptions (\ref{be}), (\ref{et}) and inequality
\[\|{\rm tr}\,(\btau) I_d\|\le d\,\|\btau\|\qquad\forall\,\btau\in\mathbb{S}^d\]
imply that
\[
(A\bu-A\bv,\bw)_X
\leq (2\beta+d\eta)\|\bu-\bv\|_X\|\bw\|_X
\]
and, therefore,
\begin{equation}
\label{A2}\|A\bu-A\bv\|_X\le(2\beta+d\eta)\|\bu-\bv\|_X.
\end{equation}
Inequalities (\ref{A1}) and (\ref{A2}) show that  the operator $A$ defined by  (\ref{Am}) satisfies condition (\ref{A})
with $m_A=2\beta$ and $L_A=2\beta+d\eta$.

Let $\bu,\,\bv\in C(\real_+;X)$ and $\bw\in X$. We use definition (\ref{Sm}) and
 the assumption (\ref{om}) to deduce that
 \begin{eqnarray*}
 &&(\cS\bu(t)-\cS\bv(t),\bw)_X
 \leq\Big(\int_0^te^{-\omega(t-s)}\|\bu(s)-\bv(s)\|_X\,ds\Big)\|\bw\|_X\\ [2mm]
 &&\qquad  \leq\Big(\int_0^t\|\bu(s)-\bv(s)\|_X\,ds\Big)\|\bw\|_X\qquad\forall\, t\in\real_+.
 \end{eqnarray*}
 This proves that
  \begin{equation*}
 \|\cS\bu(t)-\cS\bv(t)\|_X \le  \int_0^t\|\bu(s)-\bv(s)\|_X\,ds \qquad\forall\, t\in\real_+
 \end{equation*}
 which shows that  the operator  (\ref{Sm}) satisfies condition  (\ref{S}) with $s_m=1$, for each $m\in\mathbb{N}$.

Next, we claim that the form $b$  given by  (\ref{bm}) satisfies condition (\ref{b}). For the proof of this statement we refer the reader to \cite{ANZIAM}, for instance. Moreover, it is obvious to see that the operator $\pi$ defined by (\ref{pim}) satisfies condition (\ref{pi}). On the other hand,
assumption  (\ref{gm}) shows that the set $\Lambda$ defined by (\ref{Lam}) satisfies condition (\ref{lam}) and, finally, assumptions
(\ref{f0}), (\ref{f2})  imply (\ref{f}) for the element $\fb$  given by (\ref{ff}). Recall also that condition (\ref{h})  also holds, since $h$ vanishes. Therefore, Theorem \ref{t6} is now a direct consequence of Proposition \ref{pr2}. \end{proof}

\medskip A pair  $(\bu,\blambda)\in C(\real_+;X\times \Lambda)$ which satisfies (\ref{m1}) and  (\ref{m2}) for each $t\in\real_+$ is called a weak solution to Problem \ref{Pm}. We conclude from here that Theorem \ref{t6} provides sufficient conditions which guarantee the weak solvability of the contact problem (\ref{1m})--(\ref{5mm}).

\medskip We now study the continuous dependence of the solution to Problem \ref{pv} with respect to the data. To this end we assume that the density of body forces and traction are such that

\begin{eqnarray}
&&\label{f0m} \fb_0(t)=\theta(t)\wfb_0\qquad \forall\, t\in\real_+,\\
&&\label{f2m}\fb_2(t)=\zeta(t)\wfb_2\qquad \forall\, t\in\real_+
\end{eqnarray}
where  the functions $\theta\in C(\real_+;\real)$ and $\zeta\in C(\real_+;\real)$ are given and, moreover,
\begin{eqnarray}
&&\label{f0n} \wfb_0\in L^2(\Omega)^d,\\
&&\label{f2n} \wfb_2\in L^2(\Gamma_2)^d.
\end{eqnarray}
Note that in this case conditions (\ref{f0}) and (\ref{f2}) are satisfied.
We also consider the product space
$W=\mathbb{R}^3\times L^2(\Omega)^d\times L^2(\Gamma_2)^d\times\mathbb{R}$ endowed with the
canonical Hilbertian norm
and let
${U}$ be the subset of $W$ defined by
\begin{equation}
\label{U} {U}=\{\,p=(\beta,\eta,\omega,\wfb_0,\wfb_2,g)\in W\ :\ \beta,\, \eta,\, \omega,\, g >\delta_0\,\}
\end{equation}
where $\delta_0>0$ is given.
Moreover, for each $p=(\beta,\eta,\omega,\fb_0,\fb_2,g)\in {U}$ we denote by $(\bu_p,\blambda_p)$  the solution of Problem $\ref{pv}$ obtained in Theorem $\ref{t6}$.
Then, we have the following convergence result.

\begin{theorem}\label{t7}
	For each sequence $\{p_n\}\subset {U}$ such that
	${p_n} \rightharpoonup p\ {\rm in}\ W$,
	and for each $t\in \real_+$, the following convergences hold:
	\begin{eqnarray}
	&&\label{cvum}\bu_{p_n}(t)\to \bu_p(t)\qquad{\rm in}\quad X,\\ [2mm]
	&&\label{cvlam}\blambda_{p_n}(t)\rightharpoonup \blambda_p(t)\qquad\ {\rm in}\quad Y.
	\end{eqnarray}
\end{theorem}

\begin{proof}
Let $\{p_n\}\subset {U}$ be a sequence of elements in ${U}$ such that
\[p_n=(\beta_n,\eta_n,\omega_n,\wfb_{0n},\wfb_{2n}, g_n).\]
For each $n\in\mathbb{N}$,
let $A_n:X\to X$, $\cS_n:C(\real_+;X)\to C(\real_+;X)$, $\Lambda_n\subset Y$ and $\fb_n:\real_+\to Z$  be defined by equalities
\begin{eqnarray*}
&&\label{Amn} (A_n\bu,\bv)_X=\int_\Omega \Big(2\beta_n\bvarepsilon({\bu})+\eta_n\, {\rm tr}(\bvarepsilon({\bu}))I_d\Big)\cdot\bvarepsilon(\bv)\,dx\quad\forall\, \bu,\, \bv\in
X,\\[4mm]
&&\label{Smn} (\cS_n\bu(t),\bv)_X=\int_\Omega\Big(\int_0^te^{-\omega_n(t-s)}\bvarepsilon({\bu}(s))\,ds\Big)\cdot\bvarepsilon(\bv)\,dx\\[4mm]
&&\qquad\qquad\qquad\qquad\qquad\forall\, \bu\in C(\real_+;X),\ \bv\in
X,\nonumber \\[4mm]
&&\label{Lamn}\Lambda_n=\Big\{\bmu\in
Y\,:\,\langle\bmu,\,\bxi\rangle\leq g_n
\int_{\Gamma_3}\|\bxi\|da\quad \forall\, \bxi\in
\gamma(X)\Big\},\\[4mm]
&&\label{ffn}\fb_n(t)=(\theta(t)\wfb_{0n},\zeta(t)\wfb_{2n})\qquad\forall\, t\in\real_+.
\end{eqnarray*}
Moreover, for simplicity, denote  $\bu_{p_n}=\bu_n$ and $\blambda_{p_n}=\blambda_n$.
Then, it follows that $(\bu_n,\blambda_n)\in C(\real_+ ;X\times \Lambda_n)$ and, in addition,
\begin{eqnarray}
&&\hspace{-10mm}\label{m1n}(A_n\bu_n(t),\bv)_{X}+(\cS_n\bu_n(t),\bv)_{X}+b(\bv,\blambda_n(t))=(\fb_n(t),\,\pi\bv)_Z \quad\forall\, \bv\in X,\\ [2mm]
&&\hspace{-10mm}\label{m2n}b(\bu_n(t),\bmu-\blambda_n(t))\le 0\qquad       \forall\,\bmu\in
\Lambda_n.
\end{eqnarray}

Assume now that
\[p_n=(\beta_n,\eta_n,\omega_n,\wfb_{0n},\wfb_{2n}, g_n)\rightharpoonup p=(\beta,\eta,\omega,\wfb_0,\wfb_2,g)\quad{\rm in}\quad W,\]
which implies that
\begin{eqnarray}
&&\label{ww0}\beta_n\to \beta,\\ [2mm]
&&\label{ww00}\eta_n\to \eta,\\ [2mm]
&&\label{ww1}\omega_n\to \omega,\\ [2mm]
&&\label{ww2}\wfb_{0n}\rightharpoonup \wfb_0\quad{\rm in}\ \quad L^2(\Omega)^d,\\ [2mm]
&&\label{ww2n}
\wfb_{02}\rightharpoonup \wfb_2\quad{\rm in}\ \quad L^2(\Gamma_2)^d,\\ [2mm]
&&\label{ww3}g_n\to g.
\end{eqnarray}

Our aim in what follows is to apply Theorem \ref{t1} in the study of the mixed variational problems (\ref{m1n})--(\ref{m2n}) and (\ref{m1})--(\ref{m2}) and, to this end, in what follows we check the validity of
conditions (\ref{cvA})--(\ref{cvvh}).

First, we use the convergences (\ref{ww0}), (\ref{ww00}),  (\ref{ww1})  to see that condition (\ref{cvA})  is satisfied with
$F_n= 2\,|\beta_n-\beta|+ d\,|\eta_n-\eta|$ and $\delta_n=0$.
Moreover, equalities $m_n=2\beta_n$ and $L_n=2\beta_n+d\eta_n$
show that the condition (\ref{bA}) holds, too.

Let $n,\,m\in\mathbb{N}$, $t\in[0,m]$ and let  $\bv\in C(\real_+;V)$. Then, using (\ref{Smn}) and (\ref{Sm}) it is easy to see that
\begin{eqnarray}
&&\hspace{-6mm}\|\cS_n\bv(t)-\cS\bv(t)\|_X\le\Big|\Big(\int_0^te^{-\omega_n(t-s)}-e^{-\omega(t-s)}\Big)\,ds\Big|\displaystyle\Big(\max_{s\in[0,m]}\|\bv(s)\|_X\Big)\label{z3}\\ [3mm]
&&\hspace{-6mm}\quad\le \int_0^t|e^{-\omega_n(t-s)}-e^{-\omega(t-s)}|\,ds\,\displaystyle\Big(\max_{s\in[0,m]}\|\bv(s)\|_X\Big).\nonumber
\end{eqnarray}
Moreover, using the mean value theorem we deduce that for all $s\in[0,t]$ there exists $\xi_n(s)\ge0$ such that  following inequality  holds,
\[
|e^{-\omega_n(t-s)}-e^{-\omega(t-s)}|\le e^{-\xi_n(s)}(t-s)|\omega_n-\omega|
\]
Using now the inequality
\[ e^{-\xi_n(s)}(t-s)|\omega_n-\omega|\le m\,|\omega_n-\omega|\]
we deduce that
\begin{equation}\label{z4}
|e^{-\omega_n(t-s)}-e^{-\omega(t-s)}|\le m |\omega_n-\omega|.
\end{equation}
Finally, we combine (\ref{z3}) and (\ref{z4}) to deduce that condition (\ref{cvS}) holds with $F_n^m=m|\omega_n-\omega|$ and $\delta_n^m=0$.
Recall that, the proof of Theorem \ref{t6} reveals that $s_m^n=1$ for all $n,\, m\in \mathbb{N}$ and, therefore, condition (\ref{Z}) holds, too.

Next, conditions  (\ref{cvb}) and (\ref{bb}) are obviously satisfied since in our case $b_n$=$b$  for each $n\in\mathbb{N}.$
On the other hand, the compactness of the embedding $X\subset L^2(\Omega)^d$ combined with the compactness of the trace operator $\gamma_2:X\to L^2(\Gamma_2)^d$ shows that  the operator  (\ref{pim}) satisfies condition (\ref{cvpi}). Moreover, we note that $\Lambda_n=\frac{g_n}{g}\Lambda$.
Using now
the convergence (\ref{ww3}) it  is easy to see that  condition (\ref{cvLa}) holds.
In addition, we note that the convergences (\ref{ww2}) and  (\ref{ww2n}),  together with (\ref{f0m}) and (\ref{f2m}), imply (\ref{cvvf}) for $\fb_n$ and $\fb$ given by (\ref{ffn}) and (\ref{ff}), respectively.  Finally, note that, obviously, condition (\ref{cvvh}) is satisfied.

In follows from above that we are in   position to use Theorem \ref{t1} in order to deduce  that  the convergences (\ref{cvum}) and (\ref{cvlam}) hold for each $t\in\real_+$,
which concludes the proof.
\end{proof}


Besides the mathematical interest, the convergence results  (\ref{cvum}) and (\ref{cvlam}) are important from mechanical point of view since they provide   the continuous dependence of the weak solution of Problem \ref{Pm} with respect to the Lam\'e coefficients, the relaxation coefficient, the densities of the body forces and surface tractions, and the friction bound, at each time moment.

We now provide two examples of optimization problems associated to  Problem \ref{pv} for which the abstract result in Theorem \ref{t2} holds.  Everywhere below ${U}$ represents the set given by (\ref{U}).
The two problems we consider below
have a common feature and can be casted in the following general form.

\begin{problem}\label{opm}  Given $t\in\real_+$, find  $p^*\in U$ such that
	\begin{equation}\label{o1m}
	{\cal L}(\bu_{p^*}(t),\blambda_{p^*}(t),p^*)=\min_{p\in U} {\cal L}(\bu_p(t),\blambda_p(t),p).
	\end{equation}
\end{problem}

Here ${\cal L}:X\times Y\times U \to\mathbb{R}$ is the cost functional.
Both $U$ and ${\cal L}$ will change from example to example and, therefore,  will be described below. We also recall  that, given $p=(\beta,\eta,\omega,\wfb_0,\wfb_2,g)\in U$,
$(\bu_p,\blambda_p)$ represents
the solution of  Problem $\ref{pv}$ with the data $\beta$, $\eta$, $\omega$, $\fb_0$, $\fb_2$ and $g$ where $\fb_0$, $\fb_2$  are given by (\ref{f0m}) and (\ref{f2m}), respectively.
Note that the existence of this solution is guaranteed by
Theorem \ref{t6}. Moreover, it follows from the proof of Theorem
\ref{t7} that if $p_n\rightharpoonup p$ in $W$ then conditions $(\ref{cvA})$--$(\ref{cvLa})$ hold and, obviously, $(\ref{cvvf})$ and $(\ref{cvvh})$ hold, too. Therefore, the solvability of Problem \ref{opm}
follows from Theorem \ref{t2}, provided that conditions
$(\ref{o2})$, $(\ref{o3})$  and either
$(\ref{o4})$ or $(\ref{o5})$ are satisfied.

\medskip

\begin{example}\label{eo1}
	Let $\delta_1$, $\delta_2$,  $M_0$, $M_2$ be positive constants such that $\delta_1\le\delta_2$,  and consider a  function $\bu_0\in X$. Let
	$U$ and ${\cal L}:X\times Y\times U \to\mathbb{R}$ be defined by
	\begin{eqnarray*}
		&&U=\{\,p=(\beta,\eta,\omega,\wfb_0,\wfb_2,g)\in \widetilde{U}\ :\ \beta,\,\eta,\,\omega,\,g\in[\delta_1,\delta_2],\\
		&&\qquad\qquad \  \|\fb_0\|_{L^2(\Omega)^d}\le M_0,\ \|\fb_2\|_{L^2(\Gamma_2)^d}\le M_2\,\},\\ [2mm]
		&&{\cal L}(\bu,\blambda,p)=\int_{\Gamma_3}\|\bu-\bu_0 \|^2\,da\qquad\forall\, \bu\in X,\ \blambda\in\Lambda,\ p\in U.
	\end{eqnarray*}

	With this choice, the mechanical interpretation of  Problem $\ref{opm}$ is the following:
	given a contact process of the form $(\ref{1m})$--$(\ref{5mm})$, $(\ref{f0m})$--$(\ref{f2m})$ and a time moment $t\in\real_+$,
	we are looking for a set of data  $p=(\beta,\eta,\omega,\wfb_0,\wfb_2,g)\in U$
	such that
	the corresponding displacement on the contact surface  at $t$ is as close as possible to the ``desired  displacement" $\bu_0$.
	Note that in this case assumptions $(\ref{o2})$, $(\ref{o3})$
	and $(\ref{o5})$ are satisfied.  Therefore,  Theorem $\ref{t2}$ guarantees the existence of solutions to the  corresponding optimization  problem $\ref{opm}$.
	
\end{example}


\begin{example}\label{eo2}
	Let $\bu_0\in X$ and $\blambda_0\in Y$ be given,  and let $c_1$, $c_2$, $c_3$ be strictly positive constants. Moreover,
	consider the set $U$ defined by $(\ref{U})$ and the cost functional ${\cal L}:X\times Y\times U \to\mathbb{R}$  defined by
	\begin{eqnarray*}
		&&{\cal L}(\bu,\blambda,p)=c_1\|\bu-\bu_0\|_X^2+
		c_2\|\blambda-\blambda_0\|_Y^2+ c_3\|p\|_{W}^2\\
		&&\qquad\qquad\qquad\forall\, \bu\in X,\ \blambda\in\Lambda,\ p\in U.\nonumber
	\end{eqnarray*}

	With this choice, the mechanical interpretation of  Problem $\ref{opm}$ is the following:
	given a contact process of the form $(\ref{1m})$--$(\ref{5mm})$,  $(\ref{f0m})$--$(\ref{f2m})$ and a time moment $t\in\real_+$
	we are looking for a set of data  $p=(\beta,\eta,\omega,\wfb_0,\wfb_2,g)\in U$
	such that
	the corresponding state of the body  at $t$ is as close as possible to the ``desired  state" $(\bu_0,\blambda_0)$.
	Furthermore, this choice has to fulfill a minimum expenditure condition which is taken into account by the last  term in the functional ${\cal L}$. In fact, a compromise policy between the two aims (``$\bu$ close to $\bu_0$", ``$\blambda$ close to $\blambda_0$" and ``minimal data $p$") has to be found and the relative importance of each criterion with respect to the other is expressed by the choice of the weight coefficients $c_1$, $c_2$ and $c_3$.
	Note that in this case assumptions $(\ref{o2})$, $(\ref{o3})$
	and $(\ref{o4})$ are satisfied.  Therefore,  Theorem $\ref{t2}$ guarantees the existence of the solutions to the  corresponding  optimization problem   $\ref{opm}$.
	
\end{example}

\section*{Acknowledgements}

\indent This project has received funding from the European Union's Horizon 2020
Research and Innovation Programme under the Marie Sklodowska-Curie
Grant Agreement No. 823731 CONMECH.

\end{document}